%% file: main.tex
\documentclass[final,onefignum,onetabnum]{siamart190516}

\usepackage{braket,amsfonts}

\usepackage{array}

\usepackage[caption=false]{subfig}
\usepackage{pgfplots}

\newsiamthm{claim}{Claim}
\newsiamremark{remark}{Remark}
\newsiamremark{hypothesis}{Hypothesis}
\crefname{hypothesis}{Hypothesis}{Hypotheses}

\usepackage{algorithmic}

\usepackage{graphicx,epstopdf}

\Crefname{ALC@unique}{Line}{Lines}

\usepackage{amsopn}

\usepackage{graphicx}

\usepackage{amssymb, pifont, color, graphics, amsmath, amssymb, mathrsfs, amsfonts, multicol, amsmath, tikz, pgfplots, graphicx, lmodern} 
\usepackage{tikz}
\usepackage{hyperref,stackengine}
\hypersetup{
    colorlinks=true,
    linkcolor=blue,
    filecolor=red,      
    urlcolor=blue,
}

\usepackage{mathtools}
\usepackage{esint}

\makeatletter
\newcommand{\leqnomode}{\tagsleft@true\let\veqno\@@leqno}
\newcommand{\reqnomode}{\tagsleft@false\let\veqno\@@eqno}
\makeatother
\usepackage[utf8]{inputenc}

\DeclareMathOperator*{\du}{d\!}

\DeclareMathOperator{\bu}{{\boldsymbol{u}}}
\DeclareMathOperator{\blf}{\boldsymbol{f}}
\DeclareMathOperator{\bw}{\boldsymbol{w}}
\DeclareMathOperator{\bz}{\boldsymbol{z}}
\DeclareMathOperator{\bv}{\boldsymbol{v}}

\DeclareMathOperator{\bn}{\boldsymbol{n}}

\DeclareMathOperator{\bphi}{\boldsymbol{\varphi}}
\DeclareMathOperator{\bpsi}{\boldsymbol{\psi}}

\DeclareMathOperator{\ws}{\mathrel{\ensurestackMath{\stackon[1pt]{\rightharpoonup}{\scriptstyle\ast}}}}

\DeclareMathOperator{\chiarrow}{\mathrel{\ensurestackMath{\stackon[1pt]{\longrightarrow}{\chi}}}}

\makeatletter

\newlength\normalparindent
\newenvironment{pethau}%
{\begin{list}{}%
    {%
      \setlength{\itemindent}{-2pt}%
      \setlength{\leftmargin}{20pt}%
      \setlength{\labelwidth}{.3\normalparindent}%
      \addtolength{\topsep}{-0.5\parskip}%
      \listparindent \normalparindent
      \setlength{\parsep}{\parskip}}}%
  {\end{list}}
\makeatother

\usepackage{xspace}
\usepackage{bold-extra}
\usepackage[most]{tcolorbox}

\colorlet{texcscolor}{blue!50!black}
\colorlet{texemcolor}{red!70!black}
\colorlet{texpreamble}{red!70!black}
\colorlet{codebackground}{black!25!white!25}

\lstdefinestyle{siamlatex}{%
  style=tcblatex,
  texcsstyle=*\color{texcscolor},
  texcsstyle=[2]\color{texemcolor},
  keywordstyle=[2]\color{texemcolor},
  moretexcs={cref,Cref,maketitle,mathcal,text,headers,email,url},
}

\tcbset{%
  colframe=black!75!white!75,
  coltitle=white,
  colback=codebackground, 
  colbacklower=white, 
  fonttitle=\bfseries,
  arc=0pt,outer arc=0pt,
  top=1pt,bottom=1pt,left=1mm,right=1mm,middle=1mm,boxsep=1mm,
  leftrule=0.3mm,rightrule=0.3mm,toprule=0.3mm,bottomrule=0.3mm,
  listing options={style=siamlatex}
}

\newtcblisting[use counter=example]{example}[2][]{%
  title={Example~\thetcbcounter: #2},#1}

\newtcbinputlisting[use counter=example]{\examplefile}[3][]{%
  title={Example~\thetcbcounter: #2},listing file={#3},#1}

\DeclareTotalTCBox{\code}{ v O{} }
{ 
  fontupper=\ttfamily\color{black},
  nobeforeafter,
  tcbox raise base,
  colback=codebackground,colframe=white,
  top=0pt,bottom=0pt,left=0mm,right=0mm,
  leftrule=0pt,rightrule=0pt,toprule=0mm,bottomrule=0mm,
  boxsep=0.5mm,
  #2}{#1}

\patchcmd\newpage{\vfil}{}{}{}
\begin{tcbverbatimwrite}{tmp_\jobname_header.tex}
\title{Long-Time Behaviour of Shape Design Solutions for the Navier--Stokes Equations}

\author{John Sebastian H. Simon\thanks{Institute of Mathematics, 
              Czech Academy of Science, 
              {{\v Z}itn{\'a} 25 115 67 Praha 1}, {Czech Republic}} \\(\email{simon@math.cas.cz},\email{jhsimon1729@gmail.com}).} 

\headers{Long-Time Behaviour of Shape Design Solutions for the Navier--Stokes Equations}{John Sebastian H. Simon}
\end{tcbverbatimwrite}
\input{tmp_\jobname_header.tex}

\ifpdf
\hypersetup{ pdftitle={Navier-Stokes} }
\fi

\begin{document}
\maketitle

\begin{tcbverbatimwrite}{tmp_\jobname_abstract.tex}
\begin{abstract}
  We investigate the behavior of dynamic shape design problems for fluid flow at large time horizon.  In particular, we shall compare the shape solutions of a dynamic shape optimization problem with that of a stationary problem and show that the solution of the former approaches a neighborhood of that of the latter. The convergence of domains is based on the $L^\infty$-topology of their corresponding characteristic functions which is closed under the set of domains satisfying the {\it cone property}.  As a consequence, we show that the asymptotic convergence of shape solutions for parabolic/elliptic problems is a particular case of our analysis. Lastly, a numerical example is provided to show the occurrence of the convergence of shape design solutions of time-dependent problems with different values of the terminal time $T$ to a shape design solution of the stationary problem. 
\end{abstract}


\end{tcbverbatimwrite}
\input{tmp_\jobname_abstract.tex}

\section{Introduction}
Shape design problems for fluid flow control captivate a vast number of mathematicians as well as engineers due to its applications in aeronautics, optimal mixing problems, and fluid-structure interaction problems, to name a few. Such problems are mostly formulated so as to optimize a given objective function constrained with the stationary fluid model, e.g. the stationary Stokes and Navier-Stokes equations. Among these studies are \cite{pironneau1974}, and \cite{pironneau2010} where the drag around a body in a fluid is minimized for viscous fluids. Z. Gao, et al. \cite{gao2008} on the other hand determined the gradient of the objective functional using three different methods, such as the use of Piola transform, minimax formulation, and the function space embedding technique. Aside from the mentioned references there is a pool of literature dealing with shape design problems constrained with the stationary fluid equations (see for example \cite{garcke2018}, \cite{haslinger2005}, \cite{haslinger2017}). In fact, majority of shape design problems involving fluid flow is governed by stationary equations, while there are only quite a few literature with the time-dependent case (see \cite{brandenburg2009}, \cite{yagi2005}, \cite{yagi2007}, \cite{gao2010}). 

One of the reasons for the disparity in the number of sources between the time-dependent and stationary problems is perhaps due to computational complexity of solving the dynamic systems. Aside from that, the assumption is that for long time horizons, the dynamic optimization problem generates a closely similar shape to that of the stationary problem. Such property has been proven to hold true for a lot of control problems and has been popularized as the {\it turnpike} property. The said property states that for large time optimal control problems,  the solutions may be divided into three periods -- the first and last periods are known to be short-time periods, and the middle part is known to steer the solutions (both the state and the controls)  to be exponentially close to the solutions of the stationary problem. See \cite{zaslavski2015, zaslavski2006} and the references therein for an extensive review of the turnpike property.

In the case of fluid flow control problems, a very few literature can be cited for the turnpike property. In fact, until 2018, the problem of proving the turnpike property for an optimal control problem constrained with the Navier--Stokes equations is open. S. Zamorano \cite{zamorano2018} proved the occurrence of turnpike to the control and the states. In the said reference, the author considered cases where the control is both time-dependent and independent. 

For shape design problems, on the other hand, turnpike property is mostly an open problem. Nevertheless, G. Lance, et al. \cite{lance2019} proved a weaker notion of the turnpike property and numerically illustrated that such phenomenon occurs for a shape optimization problem constrained with the heat equation for the time-dependent problem and the Poisson equation for the equilibrium problem.  

Another seminal result in long time behavior of shape design solutions is done by E. Trelat, et al. \cite{trelat2017}, where the authors showed that for the heat and Poisson equations, as the time horizon gets larger the shape solutions of the dynamic optimization problem asymptotically converges to that of the stationary problem. Furthermore, the authors established that the limiting domain - as the terminal time approaches infinity - converges to a domain that solves the stationary optimization problem. {We also mention the work done by G. Allaire et al., \cite{allaire2010} where they proved the convergence of solutions of optimal design problems constrained with the heat equation to a solution of an optimal design problem constrained with the stationary version of the said state. The optimal design problem was proposed to determine the best optimal way of arranging two conducting materials according to some physical properties. }

In this short note, we shall investigate the long time behavior of the solutions to shape optimization problems governed by the Navier-Stokes equations. Specifically, we shall study the asymptotic convergence of the solution of the shape optimization problem with instationary Navier--Stokes equations to the solution of the problem constrained with the stationary Navier--Stokes equations. To simplify the analysis, we assume that the fluid source function is not dependent on the time variable, and the optimization problem is formulated so as to steer the flow - in terms of the velocity and its gradient - into a prescribed profile. 

As opposed to the method used in \cite{trelat2017}, we shall not resort to the Hausdorff complement topology and the concept of $\Gamma$-convergence of domains, but instead we shall use the $L^\infty$-topology on the indicator functions of the domains and define the convergence of domains based on such topology. {Such topology was also used in \cite{munch2008} where the authors studied an optimal design problem that is set up the same way with that of \cite{allaire2010}. In particular the $L^\infty$-topology is used for the relaxed version of the optimization problem they considered.} Lastly, we shall show that when reduced to the Stokes equations, our results reflect the same results of E. Trelat et al.,\cite{trelat2017}.

Our exposition will be as follows: in the next section we shall introduce the shape optimization problems, and their governing state equations. The said state equations will be modified so as to take into account the possibility of simplifying them into the Stokes equations. Section \ref{section:3} is dedicated to the existence and uniqueness of solutions to the fluid equations, while we prove existence of shape solutions in Section \ref{section:4}.  In Section \ref{section:5}, we will prove our main results, then we shall provide a numerical illustration in Section \ref{section:6} where we will utilize traction methods for solving the deformation of domains. We leave some concluding remarks in the last section, as well as possible future directions.

\section{Shape Design Problems}
\label{section:2}
Let $\mathcal{D}\subset\mathbb{R}^2$ be a non-empty open bounded connected domain, and $\omega \subsetneq \mathcal{D}$, we consider the following set of admissible domains
\begin{align*}
	\mathcal{O}_{\omega} = &\left\{\Omega\subset \mathcal{D}: \Omega\supset\omega,\ \Omega\text{ is open, bounded, connected, and at least of class }C^{1,1} \right\}.
\end{align*}
We shall also consider a static fluid external force ${\blf}\in L^2(\mathcal{D};\mathbb{R}^2)$ for both stationary and time-dependent Navier--Stokes equations. In fact, for a given $\Omega\in \mathcal{O}_{\omega}$, we shall consider a dynamic equation on the interval $(0,T)$ given by
\begin{align}
	\left\{
	\begin{aligned}
		\partial_t{\bu} - \nu\Delta{\bu} + \gamma({\bu}\cdot\nabla){\bu} + \nabla p & = {\blf} && \text{in }\Omega\times (0,T),\\
		\nabla\cdot{\bu} & = 0 &&\text{in }\Omega\times (0,T),\\
		{\bu} & = 0 &&\text{in }\partial\Omega\times (0,T),\\
		{\bu} & = {\bu}_0 &&\text{in }\Omega\times\{0\}.
	\end{aligned}
	\right.
	\label{system:timedependent}
\end{align}
where ${\bu}$ and $p$ correspond to the dynamic fluid velocity and pressure, respectively,  and ${\bu}_0\in L^2(\mathcal{D};\mathbb{R}^2)$ is the initial velocity that satisfies $\nabla\cdot{\bu}_0 = 0$ in $\Omega$. On the other hand, the stationary Navier--Stokes equation is given by
\begin{align}
	\left\{
	\begin{aligned}
		- \nu\Delta{\bv} + \gamma({\bv}\cdot\nabla){\bv} + \nabla q & = {\blf} && \text{in }\Omega,\\
		\nabla\cdot{\bv} & = 0 &&\text{in }\Omega,\\
		{\bv} & = 0 &&\text{in }\partial\Omega.
	\end{aligned}
	\right.
	\label{system:stationary}
\end{align}
where ${\bv}$ and $q$ are the respective equilibrium fluid velocity and pressure. On both equations, $\nu>0$ denotes the fluid viscosity, and we call the parameter $\gamma\ge 0$ the convection constant.
We note that if $\gamma = 1$, we are dealing with the usual Navier--Stokes equations, while when $\gamma = 0$ then both states are reduced to the Stokes equation. 

Our intent is focused on analyzing two shape optimization problems governed by equations \eqref{system:timedependent} and \eqref{system:stationary}. In particular, for a given static desired velocity ${\bu}_D \in L^2(\omega;\mathbb{R}^2)$, we consider the time average problem given by
\begin{align}
	\left.
	\begin{aligned}
		\min_{\Omega\in \mathcal{O}_{\omega}} J_T(\Omega) := \frac{\nu}{T}&\int_0^T \|{\bu}(t) - {\bu}_D\|_{L^2(\omega;\mathbb{R}^2)}^2 + \|\nabla({\bu}(t) - {\bu}_D)\|_{L^2(\omega;\mathbb{R}^{2\times 2})}^2 \du t\\ \text{subject to }\eqref{system:timedependent},
	\end{aligned}
	\right.
	\label{problem:timedependent}
\end{align}
and the stationary shape design problem given by 
\begin{align}
	\left.
	\begin{aligned}
		\min_{\Omega\in \mathcal{O}_{\omega}} J_s(\Omega) :=\nu\big(\|{\bv} &- {\bu}_D\|_{L^2(\omega;\mathbb{R}^2)}^2 + \|\nabla({\bv} - {\bu}_D)\|_{L^2(\omega;\mathbb{R}^{2\times2})}^2\big)\\ \text{subject to }\eqref{system:stationary}.
	\end{aligned}
	\right.
	\label{problem:stationary}
\end{align}

Our goal - assuming for now that \eqref{problem:timedependent} and \eqref{problem:stationary} are well-posed, with $\Omega_T$ and $\Omega_s$ being their respective solutions - is to show that {
	\begin{align}
		|J_T^* - J_s^*| \le c \left( \frac{1}{T} + \frac{1}{\sqrt{T}} \right),
		\label{ineq:goal}
\end{align} }
where $J_T^* := J_T(\Omega_T)$, $J_s^*:= J_s(\Omega_s)$,  and the constant $c := c( {\bu}_0,{\bu}_D, {\blf}, 1/\nu, \mathcal{D} ) >0$ is independent of $T$. 

\begin{remark}
	Note that when $\gamma = 0$, both systems \eqref{system:timedependent} and \eqref{system:stationary} become the Stokes equations, and can be realized as parabolic and elliptic problems. In fact,  the asymptotic convergence coincides with that in \cite{trelat2017}.
\end{remark}

\section{Existence and uniqueness of state solutions}
\label{section:3}

In this section, we show the existence of the solutions - in weak sense - of systems \eqref{system:timedependent} and \eqref{system:stationary}. We begin by introducing the necessary function spaces. Let $X$ and $Y$ be normed spaces. We denote by $\mathcal{L}(X;Y)$ the set of continuous linear operators from $X$ to $Y$ with the norm $\|f\|_{\mathcal{L}(X;Y)} = \sup_{x\in X\backslash\{0\}}\frac{\|f(x)\|_Y}{\|x\|_X}$ for $f\in\mathcal{L}(X;Y)$.
The following spaces were already used in the previous section, but for the sake of completeness, we define them formally. For a measurable set $D\subset\mathbb{R}^2$, $d = 1,2,2\times2$ and $p\ge 1$, the space $L^p(D;\mathbb{R}^d)$ is the space of $p^{th}$ integrable functions from $D$ to $\mathbb{R}^d$, and the usual Sobolev spaces from $D$ to $\mathbb{R}^d$ are denoted by $W^{m,p}(D;\mathbb{R}^d)$ with $m \in\mathbb{N}$. For $p=2$, we use the notation $H^m(D;\mathbb{R}^d) = W^{m,p}(D;\mathbb{R}^d)$.  The space of the functions in $H^m(D;\mathbb{R}^d)$ whose traces on the boundary of $D$ are zero will be denoted as $H^m_0(D;\mathbb{R}^d)$ upon which the norm
\begin{align*}
	\|{\bu}\|_{H^m_0(D;\mathbb{R}^d)} = \left(\sum_{|\alpha| = m}\int_D |\partial^\alpha{\bu}|^2 \du x \right)^{\!1/2}
\end{align*}
is endowed, where $\alpha$ is a multi-index, and the partial derivatives $\partial^\alpha{\bu}$ are understood in the sense of distributions. 

For a domain $\Omega\in \mathcal{O}_{\omega}$, to take into account the divergence-free property of the fluid velocities, we consider the following spaces 
\begin{align*}
	& V(\Omega) := \{{\bu}\in H^1_0(\Omega;\mathbb{R}^2): \nabla\cdot{\bu} = 0\text{ in }L^2(\Omega;\mathbb{R})  \},\\
	& H(\Omega) := \{{\bu}\in L^2(\Omega;\mathbb{R}^2): \nabla\cdot{\bu} = 0\text{ in }L^2(\Omega;\mathbb{R}),{\bu}\cdot{\bn} = 0 \text{ on }\partial\Omega  \}.
\end{align*}

We denote by $V^*(\Omega)$ the dual space of $V(\Omega)$, with the norm 
\[
\|{\bu}\|_{V^*(\Omega)} = \sup_{{\bv} \in V(\Omega)\backslash\{0\}}\frac{{}_{V^*(\Omega)}\langle {\bu},{\bv}\rangle_{V(\Omega)}}{\|{\bv}\|_{V(\Omega)}},
\]
where ${}_{V^*(\Omega)}\langle {\bu},{\bv}\rangle_{V(\Omega)}$ corresponds to the duality pairing of elements of $V^*(\Omega)$ and $V(\Omega)$. For the pressure term, we consider the space $L^2_0(\Omega;\mathbb{R}) = \{ q\in  L^2(\Omega;\mathbb{R}): \int_\Omega q\du x = 0   \}$. We also use the notation $(\cdot,\cdot)_\Omega$ for the $L^2(\Omega;\mathbb{R}^d)$ inner product, where $d = 1,2,2\times2$.

The time-dependent problems will be analyzed by virtue of the following spaces of Banach valued functions.  For a terminal time $T>0$ and a real Banach space $X$, we denote the space of continuous functions from $I = (0,T)$ to $X$ by $C(I;X)$ with the norm $\sup_{t} \|u(t) \|_X$.  The Bochner space $L^p(I;X)$ for $p\ge 1$ is also considered with the norm 
\begin{align*}
	\|u\|_{L^p(I;X)} = \left\{ \begin{aligned} 
		&\mathrm{ess}\sup_{t\in I}\|u(t)\|_X &&\text{for }p=\infty,\\
		&\left(\int_I \|u(t)\|_X^p \du t \right)^{\!1/p} &&\text{otherwise}.
	\end{aligned} \right.
\end{align*}

We also consider the space \[ W^p(\Omega):= \{{\bu} \in L^2(I;V(\Omega)); \partial_t{\bu}\in L^p(I;{V}^*(\Omega)) \}\] which is compactly embedded to $L^p(I;H(\Omega))$, and we have the following inclusion $W^2(\Omega)\subset C(\overline{I};H(\Omega))$.

Suppose that ${\bu}_0\in H(\Omega)$, we call ${\bu} \in L^\infty(I;H(\Omega))\cap W^2(\Omega)$ a {\it weak solution} of \eqref{system:timedependent} if it satisfies 
\begin{align}
	\begin{aligned}
		{}_{V^*(\Omega)}\langle \partial_t{\bu}(t), {\bphi}\rangle_{V(\Omega)} + \nu(\nabla{\bu}(t),\nabla{\bphi})_\Omega + \gamma(({\bu}(t)\cdot\nabla){\bu}(t),{\bphi})_{\Omega} &= ({\blf},{\bphi})_{\Omega}
	\end{aligned}
	\label{weak:timedependent}
\end{align}
for all ${\bphi}\in V(\Omega)$, a.e.  $t\in (0,T)$, and ${\bu}(0) = {\bu}_0$ in $H(\Omega)$. We note that the pointwise evaluation at $t = 0$ makes sense due to the inclusion $W^2(\Omega)\subset C(\overline{I};H(\Omega))$.

Meanwhile, ${\bv}\in V(\Omega)$ is a {\it weak solution} of \eqref{system:stationary} if the following equation holds true
\begin{align}
	\nu(\nabla{\bv},\nabla{\bphi})_{\Omega} + \gamma(({\bv}\cdot\nabla){\bv},{\bphi})_{\Omega} = ({\blf},{\bphi})_{\Omega},
	\label{weak:stationary}
\end{align}
for all ${\bphi}\in V(\Omega)$.

The existence of such weak solutions has been well-established (see for example \cite{girault1986},\cite{temam1977}). Nevertheless, we present the said results in the following theorem. 
\begin{theorem}
	Let $\Omega \in \mathcal{O}_{\omega}$ and ${\blf}\in L^2(\mathcal{D};\mathbb{R}^2)$.
	\begin{itemize}
		\item[i)] If ${\bu}_0\in H(\Omega)\cap L^2(\mathcal{D};\mathbb{R}^2)$, then the weak solution ${\bu} \in L^\infty(I;H(\Omega))\cap W^2(\Omega) $ of \eqref{system:timedependent} exists and satisfies the following estimates
		\begin{align}
			&\begin{aligned}
				\|{\bu}\|_{L^\infty(I;H(\Omega))} \le c_1\left(\sqrt{\frac{T}{\nu}}\|{\blf}\|_{L^2(\mathcal{D};\mathbb{R}^2)} + \|{\bu}_0\|_{L^2(\mathcal{D};\mathbb{R}^2)} \right), \label{estimate:Linf}
			\end{aligned}\\
			&\begin{aligned}
				\|{\bu}\|_{L^2((0,t);V(\Omega))}^2 \le c_2\left(\frac{t}{\nu^2}\|{\blf}\|_{L^2(\mathcal{D};\mathbb{R}^2)}^2 + \frac{1}{\nu}\|{\bu}_0\|_{L^2(\mathcal{D};\mathbb{R}^2)}^2 \right),\quad \forall t\in[0,T]\label{estimate:L2}
			\end{aligned}
		\end{align}
		where $c_1,c_2>0$ are constants independent of $\Omega$ and $T$.
		\item[ii)] There exists a weak solution ${\bv}\in V(\Omega)$ of \eqref{system:stationary} which satisfies 
		\begin{align}
			\|{\bv}\|_{V(\Omega)} \le \frac{\tilde c}{\nu}\|{\blf}\|_{L^2(\mathcal{D};\mathbb{R}^2)},
			\label{estimate:stationary}
		\end{align}
		where $\tilde c>0$ is a constant which can be chosen to be independent of $\Omega$. Furthermore, if we assume that $2^{1/2}\tilde{c}^2\gamma\|{\blf}\|_{L^2(\mathcal{D};\mathbb{R}^2)} < \nu^2$, where $c>0$ is the same as in \eqref{estimate:stationary}, then the solution is unique.
	\end{itemize}
	\label{theorem:statewp}
\end{theorem}

\begin{proof}
	{
		The proof of existence and uniqueness is a routine procedure and can be easily done. Nonetheless, we lay such proof here for completion. }
	
	{
		Let us begin with an orthonormal eigenbasis $\{{\bphi}_k \}\subset V(\Omega)$ of the space $H(\Omega)$ which consists of eigenfunctions for the Stokes operator, which are also orthogonal in $V(\Omega)$ and so that ${\bu} = \sum_{k=1}^\infty \alpha_k {\bphi}_k $ for all ${\bu}\in H(\Omega)$, where $\alpha_k$ are known as the Fourier coefficients.  Associated with such eigenfunctions are the eigenvalues $0<\lambda_1\le \lambda_2 \le \cdots \le \lambda_k \to \infty$ as $k\to\infty$ from whence we implore that
		\begin{align*}
			(\nabla{\bphi}_k,\nabla{\bv})_\Omega = \lambda_k({\bphi}_k,{\bv})_\Omega, \quad\|{\bv}\|_{H(\Omega)} = \left( \sum_{k=1}^\infty\alpha_k^2 \right)^{1/2}\text{ and }\|{\bv}\|_{V(\Omega)} = \left( \sum_{k=1}^\infty \lambda_k\alpha_k^2 \right)^{1/2} \quad \forall {\bv}\in V(\Omega).
		\end{align*}
	}
	
	{
		Let us first establish the existence of solutions to \eqref{weak:timedependent}. 
		We start with projecting such problem into the finite-dimensional space $V_m := \textrm{span}\{{\bphi}\}_{k=1}^m$.
		In particular, we solve for ${\bu}_m(t) = \sum_{k=1}^m\alpha_k(t){\bphi}_k \in V_m$ that solves
		\begin{align}
			{}_{V^*(\Omega)}\langle  \partial_t{\bu}_m(t), {\bphi}_j \rangle_{V(\Omega)}+ \nu(\nabla{\bu}_m(t),\nabla{\bphi}_j)_\Omega + \gamma(({\bu}_m(t)\cdot\nabla){\bu}_m(t),{\bphi}_j)_{\Omega} &= ({\blf},{\bphi}_j)_{\Omega}\quad \forall j=1,2,\ldots,m,
			\label{weak:discretetime}
		\end{align}
		and satisfies ${\bu}_m(0) = {\bu}_{0m} := \sum_{k=1}^m({\bu}_0,{\bphi}_k)_\Omega{\bphi}_k \in V_m$.
		Due to the orthonormality of $\{{\bphi}\}_{k=1}^m$,  \eqref{weak:discretetime} can be rewritten as the initial value problem 
		\begin{align}
			\alpha_j'(t) + \nu\lambda_j\alpha_j(t) + \gamma\sum_{l,k=1}^m \beta_{l k j}\alpha_{l}(t)\alpha_k(t) &= ({\blf},{\bphi}_j)_{\Omega},\quad \alpha_j(0) =({\bu}_0,{\bphi}_j)_\Omega, 
			\label{ode:discretetime}
		\end{align}
		for all $j=1,2,\ldots,m$, where $ \beta_{l k j} = (({\bphi}_l\cdot\nabla){\bphi}_k,{\bphi}_j	)_\Omega$ which is zero whenever $k=j$. 
		Using usual arguments for existence of solutions of initial value problems we infer the existence of $\alpha_j\in H^1([0,t_j];\mathbb{R})$ that solves \eqref{ode:discretetime}, which also implies the existence of $ {\bu}_m\in H^1([0,t_m]; V_m) $ for some $t_m>0$ . 
		If $t_m<T$ one expects $\|{\bu}_m\|_{H(\Omega)}\to +\infty$ as $t\to t_m$. 
		Thankfully, the upcoming estimate prove otherwise.}
	
	{
		The promised estimate is achieved by first multiplying both sides of \eqref{ode:discretetime} by $2\alpha_j(t)$, adding the resulting products for all $j=1,2,\ldots,m$, then integrating the resulting product over the interval $[0,t]$, for $t\le T$.  Before we show the result of such steps, we note that $2\sum_{j=1}^m\alpha_j'(t)\alpha_j(t) = \frac{d}{dt}\|{\bu}_m(t)\|_{H(\Omega)}^2$, $2\sum_{j=1}^m\alpha_j(t)\alpha_j(t)=2\|{\bu}_m(t)\|_{V(\Omega)}^2$, and $\sum_{j=1}^m\sum_{l,k=1}^m \beta_{l k j}\alpha_{l}(t)\alpha_k(t)\alpha_j(t) = 0$, and $2\sum_{j=1}^m ({\blf},{\bphi}_j)_\Omega\alpha_j(t) = 2 ({\blf},{\bu}_m(t))_\Omega$. Indeed, the first two steps implies
		\begin{align*}
			\frac{d}{dt} \|{\bu}_m(t)\|_{H(\Omega)}^2 + 2\nu \|{\bu}_m(t)\|_{V(\Omega)}^2 = 2 ({\blf},{\bu}_m(t))_\Omega.
		\end{align*}
		The last step, i.e., taking the integral over $(0,t)$, together with Young's and H{\"o}lder's inequalities imply that
		\begin{align}
			\|{\bu}_m(t)\|_{H(\Omega)}^2 + \nu \int_0^t  \|{\bu}_m(s)\|_{V(\Omega)}^2 \du s \le \|{\bu}_{0m}\|_{H(\Omega)}^2 + \frac{t}{\nu}\|{\blf}\|_{L^2(\Omega;\mathbb{R}^2)} \le \|{\bu}_{0}\|_{H(\Omega)}^2 + \frac{t}{\nu}\|{\blf}\|_{L^2(\Omega;\mathbb{R}^2)}.
			\label{estimate:finitesolution}
		\end{align}
		The last inequality is achieved from the fact that $\|{\bu}_{0m}\|_{H(\Omega)}\le \|{\bu}_{0}\|_{H(\Omega)} $. 
		Taking the supremum of both sides of \eqref{estimate:finitesolution} over the interval $[0,T]$ one gets $ \sup_{t\in [0,T]}\|{\bu}_m(t)\|_{H(\Omega)}^2\le \|{\bu}_{0}\|_{H(\Omega)}^2 + \frac{T}{\nu}\|{\blf}\|_{L^2(\Omega;\mathbb{R}^2)} $, and thus $t_m = T$. }
	
	{
		Aside from establishing that $ {\bu}_m\in H^1(I; V_m) $, \eqref{estimate:finitesolution} implies that $\|{\bu}_m\|_{L^\infty(I;H(\Omega))}$ and $\|{\bu}_m\|_{L^2(I;V(\Omega))}$ are uniformly bounded, from which we infer that 
		\begin{align}
			& {\bu}_m\ws {\bu} \text{ in }L^\infty(I;H(\Omega)),\label{conv:linf}\\
			& {\bu}_m\rightharpoonup {\bu} \text{ in }L^2(I;V(\Omega)),\label{conv:l2V}
		\end{align}
		for some ${\bu}\in L^\infty(I;H(\Omega))\cap L^2(I;V(\Omega))$.  Of course, one would expect such element to be of the form ${\bu} = \sum_{k=1}^\infty \alpha_{k}(t){\bphi}_k$ due to uniqueness of limits, the ansatz used for the element ${\bu}_m$ and the spectral representation of the spaces $V(\Omega)$ and $H(\Omega)$.  }
	
	{
		The next step is to pass the limits through \eqref{weak:discretetime}. However \eqref{conv:linf} and \eqref{conv:l2V} are insufficient due to the nonlinearity of the state equations. What we need is a stronger sense of convergence.  
		To arrive at such convergence, we use the fact that $W^2(\Omega)$ is  compactly embedded to $L^2(I;H(\Omega))$.
		So, what we need to do now is establish that the sequence $\{\partial_t{\bu}_m \}$ is uniformly bounded in $L^2(I;V^*(\Omega))$. 
		Integrating \eqref{weak:discretetime} over the interval $(0,T)$ and by using Poincare and H{\"o}lder inequalities, we see that
		\begin{align*}
			\int_0^T {}_{V^*(\Omega)}\langle  \partial_t{\bu}_m(t), {\bphi}_j \rangle_{V(\Omega)} \du t \le c\left( T\|{\blf}\|_{L^2(\Omega;\mathbb{R}^2) } + (\nu + \gamma\|{\bu}_m\|_{L^2(I;V(\Omega)) } )\|{\bu}_m\|_{L^2(I;V(\Omega)) }  \right)\|{\bphi}_j\|_{V(\Omega)}\quad \forall j=1,2,\ldots,m.
		\end{align*}
		for some constant $c>0$. Since $L^2(I;V_m)\subset L^2(I;V^*(\Omega))$, dividing both sides by $\|{\bphi}_j\|_{V(\Omega)}$ yields
		\begin{align*}
			\|\partial_t{\bu}_m\|_{L^2(I;V^*(\Omega))} \le c\left( T\|{\blf}\|_{L^2(\Omega;\mathbb{R}^2) } + (\nu + \gamma\|{\bu}_m\|_{L^2(I;V(\Omega)) } )\|{\bu}_m\|_{L^2(I;V(\Omega)) }  \right).
		\end{align*}
		From  \eqref{estimate:finitesolution}, we finally get the uniform boundedness of $\{\|\partial_t{\bu}_m\|_{L^2(I;V^*(\Omega))}\}_m$. From this, we get the following convergences:
		\begin{align}
			\partial_t{\bu}_m\rightharpoonup \partial_t{\bu}& \text{ in }L^2(I;V^*(\Omega)),\label{conv:l2Vdual}\\	
			{\bu}_m\to {\bu}& \text{ in }L^2(I;H(\Omega)),\label{conv:l2H}
	\end{align}}
	
	{
		Passing the limits on the bilinear terms are quite straightforward, in fact taking $m\to\infty$ we get that
		\begin{align*}
			{}_{V^*(\Omega)}\langle  \partial_t{\bu}_m(t), {\bphi}_j \rangle_{V(\Omega)}+ \nu(\nabla{\bu}_m(t),\nabla{\bphi}_j)_\Omega -  ({\blf},{\bphi}_j)_{\Omega} \to {}_{V^*(\Omega)}\langle  \partial_t{\bu}(t), {\bphi}_j \rangle_{V(\Omega)}+ \nu(\nabla{\bu}(t),\nabla{\bphi}_j)_\Omega -  ({\blf},{\bphi}_j)_{\Omega} 
		\end{align*}
		for a.e. $t\in (0,T)$ and the right-hand side holds for all $j\in\mathbb{N}$.}
	
	{
		What remains for us to show is that $I_m := | \int_0^T (({\bu}_m(t)\cdot\nabla){\bu}_m(t),{\bphi}_j)_{\Omega} - (({\bu}(t)\cdot\nabla){\bu}(t),{\bphi}_j)_{\Omega} \du t |\to 0$. To do this, we recall some properties of the trilinear form. Given ${\bu},{\bv},{\bw}\in V(\Omega)$, we have the following:
		\begin{align}
			&\bullet\quad (({\bu}\cdot\nabla){\bv},{\bw})_{\Omega} = - (({\bu}\cdot\nabla){\bw},{\bv})_{\Omega},\label{trilinear:commute}\\
			&\bullet\quad|(({\bu}\cdot\nabla){\bv},{\bw})_{\Omega}| \le 2^{1/2}\|{\bu}\|_{H(\Omega)}^{1/2}\|{\bu}\|_{V(\Omega)}^{1/2}\|{\bv}\|_{V(\Omega)}\|{\bw}\|_{H(\Omega)}^{1/2}\|{\bw}\|_{V(\Omega)}^{1/2},\text{ for some constant }c>0.\label{trilinear:ladyzhenskaya}
		\end{align}
		The following computation then establishes such convergence:
		\begin{align*}
			I_m \le& \left| \int_0^T ((({\bu}_m(t) -{\bu}(t)) \cdot\nabla){\bu}_m(t),{\bphi}_j)_{\Omega}  \du t \right|+ \left| \int_0^T (({\bu}(t)\cdot\nabla){\bphi}_j,{\bu}_m(t)-{\bu}(t))_{\Omega}\du t \right|\\
			& \le \|{\bphi}_j\|_{V(\Omega)} \int_0^T \|{\bu}_m(t) -{\bu}(t)\|_{H(\Omega)}^{1/2}\| {\bu}_m(t) -  {\bu}(t)\|_{V(\Omega)}^{1/2}(\|{\bu}_m(t)\|_{V(\Omega)}+\|{\bu}(t)\|_{V(\Omega)})\du t\\
			& \le c \|{\bphi}_j\|_{V(\Omega)}\|{\bu}_m - {\bu}\|_{L^2(I;H(\Omega))}^{1/2}(\|{\bu}_m\|_{L^2(I;V(\Omega))}^{3/2} + \|{\bu}\|_{L^2(I;V(\Omega))}^{3/2}).
		\end{align*}
		From the recently established strong convergence in $L^2(I;H(\Omega))$, $I_m$ indeed converges to zero as $m\to \infty$ and such convergence holds for all $j\in \mathbb{N}$.}
	
	{
		From these, we see that such ${\bu}\in L^\infty(I;H(\Omega))\cap W^2(\Omega)$ solves \eqref{weak:timedependent} for all ${\bphi}\in V(\Omega)$ and for a.e. $t\in (0,T)$. 
		Furthermore, since $W^2(\Omega)\subset C(\overline{I};H(\Omega))$, ${\bu}(0) = \lim_{m\to \infty}{\bu}_m(0)$ in $H(\Omega)$. Because $\lim_{m\to\infty}{\bu}_m(0) = \sum_{k=1}^\infty ({\bu}_0,{\bphi}_k){\bphi}_k = {\bu}_0$, we infer that ${\bu}(0) ={\bu}_0$ in $H(\Omega)$ and therefore ${\bu}\in L^\infty(I;H(\Omega))\cap W^2(\Omega)$ is a weak solution of \eqref{system:timedependent}.}
	
	{
		For the uniqueness, we assume that we have two solutions ${\bu}_1,{\bu}_2\in L^\infty(I;H(\Omega))\cap W^2(\Omega)$. The element $\tilde{\bu} = {\bu}_1-{\bu}_2 \in L^\infty(I;H(\Omega))\cap W^2(\Omega)$ satisfies the initial condition $\tilde{\bu}(0) = 0$ and the equation 
		\begin{align}
			\begin{aligned}
				{}_{V^*(\Omega)}\langle \partial_t\tilde{\bu}(t), {\bphi}\rangle_{V(\Omega)} + \nu(\nabla\tilde{\bu}(t),\nabla{\bphi})_\Omega  &=  - \gamma((\tilde{\bu}(t)\cdot\nabla){\bu}_2(t),{\bphi})_{\Omega},\quad \forall {\bphi}\in V(\Omega)
			\end{aligned}
			\label{weak:timedependentdiff}
		\end{align}
		for a.e. $t\in (0,T)$. Substituting ${\bphi} = \tilde{\bu}(t)$ into \eqref{weak:timedependentdiff} and utilizing Young inequality then yield
		\begin{align*}
			\frac{d}{dt}\|\tilde{\bu}(t)\|_{H(\Omega)}^2 + 2\nu \|\tilde{\bu}(t)\|_{V(\Omega)}^2 & \le 2^{3/2}\gamma \|\tilde{\bu}(t)\|_{H(\Omega)}\|\tilde{\bu}(t)\|_{V(\Omega)}\|{\bu}_2(t)\|_{V(\Omega)}\\
			& \le 2\nu \|\tilde{\bu}(t)\|_{V(\Omega)}^2 + c\|\tilde{\bu}(t)\|_{H(\Omega)}^2\|{\bu}_2(t)\|_{V(\Omega)}^2.
		\end{align*}
		The inequality above can then be simplified as
		\begin{align*}
			\frac{d}{dt}\|\tilde{\bu}(t)\|_{H(\Omega)}^2 \le c\|\tilde{\bu}(t)\|_{H(\Omega)}^2\|{\bu}_2(t)\|_{V(\Omega)}^2.
		\end{align*}
		Because ${\bu}_2\in L^2(I;V)$, we can use a Gronwall inequality to arrive at $\|\tilde{\bu}(t)\|_{H(\Omega)}^2 \le 0$ for all $t\in (0,T)$ which implies the uniqueness of the solution to \eqref{weak:timedependent}.}
	
	{
		We now perform diagonal testing on \eqref{weak:timedependent} by taking ${\bphi} = {\bu}(t)$.  This results to
		\begin{align*}
			\|{\bu}(t)\|_{H(\Omega)}^2 + \nu\int_0^t \|{\bu}(t)\|_{V(\Omega)}^2\du t \le c\left( \|{\bu}_0\|_{L^2(\mathcal{D};\mathbb{R}^2)}^2 + \frac{t}{\nu}\|{\blf}\|_{L^2(\mathcal{D};\mathbb{R}^2)}^2 \right).
		\end{align*}
		Focusing on the first term on the left-hand side gives us \eqref{estimate:Linf}, while the second term yields \eqref{estimate:L2}.}
	
	{
		The solution to the stationary problem \eqref{weak:stationary} can be shown to exist in a more straightforward manner. It anchors on the fact that the operator $A:V(\Omega)\times V(\Omega) \to \mathbb{R}$ defined as $\mathbb{A}({\bv},{\bphi}) = \nu(\nabla{\bv},\nabla{\bphi})_{\Omega} + \gamma(({\bv}\cdot\nabla){\bv},{\bphi})_{\Omega} $ is coercive in $V(\Omega)$, that is there exists $\alpha >0$ such that $\mathbb{A}({\bv},{\bv}) \ge \alpha\|{\bv}\|_{V(\Omega)}^2$ for all ${\bv}\in V(\Omega)$, and that $\mathbb{A}(\cdot,{\bphi})$ is weakly sequential continuous in $V(\Omega)$ for any ${\bphi}\in V(\Omega)$, i.e., if $\{{\bv}_n\}_n\subset V(\Omega)$ and ${\bv}\in V(\Omega)$ are such that ${\bv}_n\rightharpoonup {\bv}$ in $V(\Omega)$ then $\mathbb{A}({\bv}_n,{\bphi})\to \mathbb{A}({\bv},{\bphi})$. The proof of the two properties are shown below.}
	
	{
		\noindent\textbf{1. Coercivity of $\mathbb{A}$}. Since $(({\bv}\cdot\nabla){\bv},{\bv})_\Omega = 0$, we have $\mathbb{A}({\bv},{\bv}) = \nu \|{\bv}\|_{V(\Omega)}^2$. The coercivity constant is then chosen as $\alpha = \nu$.}
	
	{
		\noindent\textbf{2. Weakly sequential continuity of $\mathbb{A}(\cdot,{\bphi})$}. Let $\{{\bv}_n\}_n\subset V(\Omega)$ and ${\bv}\in V(\Omega)$ be such that ${\bv}_n\rightharpoonup {\bv}$ in $V(\Omega)$. From Rellich-Kondrachov embedding theorem, ${\bv}_n\to {\bv}$ in $H(\Omega)$.  From H{\"o}lder inequality, we thus get
		\begin{align*}
			|\mathbb{A}({\bv}_n,{\bphi}) - \mathbb{A}({\bv},{\bphi})| & \le\nu | (\nabla({\bv}_n - {\bv}),\nabla{\bphi} )_\Omega| + \gamma |(( ({\bu}_n-{\bu})\cdot\nabla ){\bu}_n,{\bphi} )_\Omega| +\gamma  |(( {\bu}\cdot\nabla ){\bphi} , {\bu}_n-{\bu})_\Omega |\\
			& \le \nu | (\nabla({\bv}_n - {\bv}),\nabla{\bphi} )_\Omega| + c\gamma(\|{\bv}_n\|_{V(\Omega)}+ \|{\bv}\|_{V(\Omega)})\|{\bv}_n - {\bv} \|_{H(\Omega)}^{1/2}\|{\bphi}\|_{V(\Omega)}.
		\end{align*}
		For a fixed ${\bphi}\in V(\Omega)$, $(\nabla(\cdot),\nabla{\bphi} )_\Omega$ is a bounded linear functional in $V(\Omega)$, hence $| (\nabla({\bv}_n - {\bv}),\nabla{\bphi} )_\Omega| \to 0$. On the other hand, since ${\bv}_n\to {\bv}$ in $H(\Omega)$, the second term on the inequality above converges to zero as well. This establishes the weak sequential continuity of $\mathbb{A}(\cdot,{\bphi})$, and consequentially the existence of the solution ${\bv}\in V(\Omega)$ to \eqref{weak:stationary}.}
	
	{
		Diagonal testing on \eqref{weak:stationary},  and H{\"o}lder and Poincare inequalities yield
		\begin{align*}
			\nu\|{\bv}\|_{V(\Omega)}^2 = ({\blf},{\bv})_\Omega \le \tilde c\|{\blf}\|_{L^2(\Omega;\mathbb{R}^2)}\|{\bv}\|_{V(\Omega)} \le \tilde c\|{\blf}\|_{L^2(\mathcal{D};\mathbb{R}^2)}\|{\bv}\|_{V(\Omega)}.
		\end{align*}
		Dividing both sides of the inequality above by $\|{\bv}\|_{V(\Omega)}$ gives us \eqref{estimate:stationary}.}
	
	{
		If ${\bv}_1,{\bv}_2\in V(\Omega)$ are solutions of \eqref{weak:stationary}, then the element $\tilde{\bv}={\bv}_1-{\bv}_2\in V(\Omega)$ solves the equation
		\begin{align*}
			\nu(\nabla\tilde{\bv},\nabla{\bphi})_\Omega + \gamma ( ({\bv}_2\cdot\nabla)\tilde{\bv}, {\bphi} )_\Omega = - \gamma( (\tilde{\bv}\cdot\nabla){\bv}_1,{\bphi} )_\Omega\quad \forall{\bphi}\in V(\Omega).
		\end{align*}
		Taking ${\bphi} = \tilde{\bv}$, using \eqref{trilinear:ladyzhenskaya}, \eqref{trilinear:commute}, Poincare inequality and \eqref{estimate:stationary}, give us
		\begin{align*}
			\nu\|\tilde{\bv}\|_{V(\Omega)}^2 \le \frac{2^{1/2}\gamma \tilde{c}^2}{\nu}\|{\blf}\|_{L^2(\mathcal{D};\mathbb{R}^2)}\|\tilde{\bv}\|_{V(\Omega)}^2.
		\end{align*}
		The inequality above can be written as $(\nu^2 - 2^{1/2}\gamma\tilde{c}^2\|{\blf}\|_{L^2(\mathcal{D};\mathbb{R}^2)})\|\tilde{\bv}\|_{V(\Omega)}^2\le 0$. The uniqueness assumption implies $\|\tilde{\bv}\|_{V(\Omega)}^2\le 0$ which concludes to the uniqueness of the solution of \eqref{weak:stationary}.
	}

\end{proof}

\begin{remark}
	(i) We note that the constant $\tilde c>0$ in \eqref{estimate:stationary} most of the time arises from Poincar{\'e} inequality, which results to it being dependent to the domain $\Omega$.  Fortunately, we can use Faber-Krahn inequality arguments to show that it indeed depends only on $\mathcal{D}$ but not on $\Omega$ (see for e.g.  \cite{bucur2005}). 
	
	(ii) A more natural assumption for the uniqueness of the stationary Navier--Stokes solution is $\mathcal{B}\gamma\tilde{c}\|{\blf}\|_{L^2(\mathcal{D};\mathbb{R}^2)} < \nu^2$, where $\mathcal{B} = {\displaystyle\sup_{{\bu},{\bv},{\bw}\in V(\Omega)}} \frac{(({\bu}\cdot\nabla){\bv},{\bw} )}{\|{\bu}\|_{V(\Omega)}\|{\bv}\|_{V(\Omega)}\|{\bw}\|_{V(\Omega)}}$. Meanwhile, from \cite[Lemma III.3.4]{temam1977} we infer that $\mathcal{B}\le 2^{1/2}\tilde c$, where $\tilde c>0$ is the Poincare constant as in \eqref{estimate:stationary}, which means that the uniqueness assumption in Theorem \ref{theorem:statewp}(ii) implies the usual assumption.
	
	{
		(iii) As can be observed, the energy estimate when $\gamma = 0$ (i.e., when reduced to the Stokes equations) is similar to the case when $\gamma >0$, it is mainly because the trilinear form when evaluated with all the same variables equates to zero. Furthermore, when $\gamma = 0$, the uniqueness assumption on the data ${\blf}\in L^2(\mathcal{D};\mathbb{R}^2)$ can be dropped. Nevertheless, it is noteworthy to point out that the nonlinear behavior caused by the convection term gives a more robust flow in the sense that the acceleration of the fluid is interpreted as not only affected by the change of the velocity through time but also by the \textit{transport} effect the fluid imposes on itself. }
	
	{
		(iv) We note that because the external force ${\blf}\in L^2(\mathcal{D};\mathbb{R}^2)$ is not dependent on time, the energy is bounded explicitly by the terminal time parameter linearly.  By a quick look, one might expect that as the terminal time grows the solution of the time-dependent problem may grow as well. One may also conjecture that because of this estimate, the energy of the gap between the solution of the time-dependent and stationary Navier--Stokes equations grows proportionally with time. However, we shall see later that such energy gap is bounded by the $L^2$-gap between the initial condition of the dynamic equation and the solution of the stationary state.
	}
	\label{remark:2}
\end{remark}

We end this section by mentioning that in the subsequent parts, the necessary assumptions for existence and uniqueness of state solutions will be assumed even when such conditions are not explicitly mentioned.

\section{Existence of shape solutions}
\label{section:4}
We show in this section that the problems \eqref{problem:timedependent} and \eqref{problem:stationary} possess solutions to exempt us from futile analyses.  We shall utilize characteristic functions of domains in $\mathbb{R}^2$, i.e., for a given nonempty domain $A\subset\mathbb{R}^2$ the characteristic function denoted as $\chi_A :\mathbb{R}^2\to [0,1]$ is defined by
\begin{align*}
	\chi_A(x) = \left\{ \begin{aligned} 1 &\qquad \text{if }x\in A,\\ 0 &\qquad \text{otherwise}. \end{aligned}\right.
\end{align*}
We shall take advantage of the $L^\infty$-topology, by which we shall define convergence of domains if their corresponding characteristic functions converge in the $L^\infty$ topology. To be precise, we say that a sequence of domains $\{\Omega_n\}\subset \mathcal{O}_{\omega}$ converges to $\Omega\in \mathcal{O}_{\omega}$ if $\chi_{\Omega_n} \to\chi_{\Omega}$ in the $L^\infty$-topology, in such case we denote the domain convergence as $\Omega_n\chiarrow\Omega$.

A requirement for a good topology on the set of admissible domains is for it to be closed under such topology.  As such, we shall utilize the so-called {\it cone property} introduced in \cite{chenais1975}.  For the definition of such property, we refer the reader to the said literature. Nevertheless, we mention important properties inferred from such condition.
\begin{lemma}
	Any element $\Omega\in\mathcal{O}_{\omega}$ satisfies the cone property. Furthermore, there exists $c>0$ such that for any $\Omega\in\mathcal{O}_{\omega}$ there exists $\mathcal{E}_\Omega^d\in\mathcal{L}(H^k(\Omega;\mathbb{R}^d);H^k(\mathcal{D};\mathbb{R}^d))$, for $d=1,2$ and $k = 0,1$, such that 
	\begin{align}
		\max_{d,k}\{\|\mathcal{E}_\Omega^d\|_{\mathcal{L}(H^k(\Omega;\mathbb{R}^d),H^k(\mathcal{D};\mathbb{R}^d))}\}\le c.
		\label{ineq:uniformbound}
	\end{align}
	Lastly, for any sequence $\{\Omega_n\}\subset\mathcal{O}_{\omega}$, there exists a subsequence $\{\Omega_{n_k}\}\subset\{\Omega_n\}$ and an element $\Omega\in\mathcal{O}_{\omega}$ such that $\Omega_{n_k}\chiarrow \Omega$. 
	\label{lemma:extensionundclosure}
\end{lemma}

\begin{proof}
	Since each $\Omega\in\mathcal{O}_{\omega}$ has a bounded and Lipschitz boundary, from \cite[Theorem 2.4.7]{henrot2018}, $\Omega$ satisfies the cone property. As for the uniform boundedness property \eqref{ineq:uniformbound} and compactness of $\mathcal{O}_{\omega}$ with respect to the $L^\infty$-topology, we refer the reader to \cite{chenais1975} and \cite[Theorem 2.4.10]{henrot2018}, respectively.   
\end{proof}

The first step in establishing existence of shape solutions is by showing the continuity of the map $\Omega\mapsto{\bphi}$ where $\bphi$ is either the solution of \eqref{weak:timedependent} or \eqref{weak:stationary}. 

\begin{proposition}
	Let $\{\Omega_n\}\subset\mathcal{O}_\omega$ be a sequence converging to an element $\Omega^*\in\mathcal{O}_\omega$. 
	\begin{itemize}
		\item[i)] For a fixed $T>0$, let ${\bu}_n\in L^\infty(I;H(\Omega_n))\cap W^2(\Omega_n)$ be the solution of \eqref{weak:timedependent} in $\Omega=\Omega_n$, and $\overline{\bu}_n := \mathcal{E}_{\Omega_n}^2({\bu}_n)$ be its extension by virtue of Lemma \ref{lemma:extensionundclosure}. Then there exists $\overline{\bu}\in  L^\infty(I;L^2(\mathcal{D};\mathbb{R}^2))\cap L^2(I;H^1(\mathcal{D};\mathbb{R}^2))$ such that $\overline{\bu}_n \rightharpoonup \overline{\bu}$ in $L^2(I;H^1(\mathcal{D};\mathbb{R}^2))$ and $\overline{\bu}_n \ws \overline{\bu}$ in $L^\infty(I;L^2(\mathcal{D};\mathbb{R}^2))$. Furthermore, ${\bu}:=\overline{\bu}|_{\Omega^*}$ solves \eqref{weak:timedependent} in $\Omega=\Omega^*$.
		\item[ii)] Suppose that ${\bv}_n \in V(\Omega_n)$ is the solution of \eqref{weak:stationary} in $\Omega = \Omega_n$, and that $\overline{\bv}_n := \mathcal{E}_{\Omega_n}^2({\bv}_n)$ is its extension in $\mathcal{D}$, then there exists $\overline{\bv}\in H^1(\mathcal{D};\mathbb{R}^2)$ such that $\overline{\bv}_n\to \overline{\bv}$ in $H^1(\mathcal{D};\mathbb{R}^2)$ and ${\bv}:= \overline{\bv}|_{\Omega^*}$ solves \eqref{weak:stationary} in $\Omega=\Omega^*$.
	\end{itemize}
	
\end{proposition}

\begin{proof}
	{\it i)} From the uniform extension property \eqref{ineq:uniformbound}, we get for each $n\in\mathbb{N}$
	\begin{align*}
		&\begin{aligned}
			\|\overline{\bu}_n\|_{L^\infty(I;L^2(\mathcal{D};\mathbb{R}^2))} \le c_1\left(\sqrt{\frac{T}{\nu}}\|{\blf}\|_{L^2(\mathcal{D};\mathbb{R}^2)} + \|{\bu}_0\|_{L^2(\mathcal{D};\mathbb{R}^2)} \right), 
		\end{aligned}\\
		&\begin{aligned}
			\|\overline{\bu}_n\|_{L^2(I;H^1(\mathcal{D};\mathbb{R}^2))}^2 \le c_2\left(\frac{T}{\nu^2}\|{\blf}\|_{L^2(\mathcal{D};\mathbb{R}^2)}^2 + \frac{1}{\nu}\|{\bu}_0\|_{L^2(\mathcal{D};\mathbb{R}^2)}^2 \right).
		\end{aligned}
	\end{align*}
	This implies that there exists $\overline{\bu}\in  L^\infty(I;L^2(\mathcal{D};\mathbb{R}^2))\cap L^2(I;H^1(\mathcal{D};\mathbb{R}^2))$ upon which $\overline{\bu}_n \rightharpoonup \overline{\bu}$ in $L^2(I;H^1(\mathcal{D};\mathbb{R}^2))$ and $\overline{\bu}_n \ws \overline{\bu}$ in $L^\infty(I;L^2(\mathcal{D};\mathbb{R}^2))$.  Furthermore, one can easily show that $\|\partial_t\overline{\bu}_n\|_{L^2(I;H^{-1}(\mathcal{D}))} \le c$ for some constant $c>0$ independent of $\Omega$ and $T>0$. Thus, we obtain the following convergences
	\begin{align}
		\left\{
		\begin{aligned}
			\overline{\bu}_n &\rightharpoonup \overline{\bu} &&\text{in }L^2(I;H^1(\mathcal{D};\mathbb{R}^2)),\\
			\overline{\bu}_n &\ws \overline{\bu} &&\text{in }L^\infty(I;L^2(\mathcal{D};\mathbb{R}^2)),\\
			\overline{\bu}_n &\to \overline{\bu} &&\text{in }L^2(I;L^2(\mathcal{D};\mathbb{R}^2)).\\
		\end{aligned}
		\right.
		\label{conv:timedependent}
	\end{align} 
	
	Since ${\bu}_n\in L^\infty(I;H(\Omega_n))\cap L^2(I;V(\Omega_n))$ solves \eqref{weak:timedependent}, then the extension satisfies 
	\begin{align*}
		(\chi_{n}\partial_t\overline{\bu}_n(t), {\bphi} )_{\mathcal{D}} + \nu(\chi_{n}\nabla\overline{\bu}_n(t),\nabla{\bphi})_{\mathcal{D}} + \gamma(\chi_{n}(\overline{\bu}(t)\cdot\nabla)\overline{\bu}(t),{\bphi})_{\mathcal{D}} &= (\chi_{n}{\blf},{\bphi})_{\mathcal{D}}, 
	\end{align*}
	for all ${\bphi}\in H^1(\mathcal{D};\mathbb{R}^2)$, where $\chi_n := \chi_{\Omega_n}$. Using the convergences in \eqref{conv:timedependent}, and the fact that $\chi_n\to\chi$ (where $\chi:=\chi_{\Omega^*}$) in $L^\infty(\mathcal{D};\mathbb{R})$ yield
	\begin{align*}
		(\chi\partial_t\overline{\bu}(t), {\bphi} )_{\mathcal{D}} + \nu(\chi\nabla\overline{\bu}(t),\nabla{\bphi})_{\mathcal{D}} + \gamma(\chi(\overline{\bu}(t)\cdot\nabla)\overline{\bu}(t),{\bphi})_{\mathcal{D}} & = (\chi{\blf},{\bphi})_{\mathcal{D}},
	\end{align*}
	for any ${\bphi}\in H^1(\mathcal{D};\mathbb{R}^2)$. By letting ${\bpsi}\in V(\Omega)$ and taking ${\bphi} = {\bpsi}$ in $\Omega^*$ and ${\bphi}=0$ in $\mathcal{D}\backslash\Omega^*$, we infer that ${\bu} = \overline{\bu}|_{\Omega^*} \in L^\infty(I;H(\Omega^*))\cap L^2(I;V(\Omega^*))$ satisfies \eqref{weak:timedependent} in $\Omega = \Omega^*$.
	
	{\it ii)} For the stationary problem, we refer the reader to \cite{kasumba2012}, among others. 
\end{proof}

\begin{remark}
	As a consequence of the strong convergence of the extensions in $H^1(\mathcal{D};\mathbb{R}^2)$, for any sequence ${\Omega_n}\subset\mathcal{O}_{\omega}$ such that $\Omega\chiarrow\Omega^*$ for some $\Omega^*\in\mathcal{O}_{\omega}$, we get the following convergence
	\begin{align}
		{\bv}_n \to \bv \text{ in }H^1(\omega;\mathbb{R}^2)
		\label{conv:strongstationary}
	\end{align}
	where ${\bv}_n\in V(\Omega_n)$ solves \eqref{weak:stationary} in $\Omega_n$, and ${\bv}\in V(\Omega^*)$ solves \eqref{weak:stationary} in $\Omega^*$.
\end{remark}

\begin{theorem}
	Suppose that the assumptions for Theorem \ref{theorem:statewp} hold, then both problems \eqref{problem:timedependent} and \eqref{problem:stationary} admit solutions.
\end{theorem}
\begin{proof}
	Note that for any $\Omega\in\mathcal{O}_\omega$, $J_T(\Omega)\ge 0$. Hence, we obtain a minimizing sequence $\{\Omega_{T,n}\}\subset\mathcal{O}_\omega$, i.e.,
	\[ 
	\lim_{n\to\infty}J_T( \Omega_{T,n} ) = \inf_{\Omega\in\mathcal{O}_\omega} J_T(\Omega) =: J_T^* .
	\]
	Since $\mathcal{O}_\omega$ is closed under the topology induced by the $L^\infty$-topology of characteristic functions, there exists a  subsequence, which we denote similarly as $\{\Omega_{T,n}\}$, such that $\Omega_{T,n}\chiarrow\Omega_T$.  Since $J_T$ is convex and continuous with respect to the state variable $\bu$, then 
	\begin{align*}
		J_T( \Omega_T) \le \liminf_{n\to\infty} J_T( \Omega_{T,n}) = J_T^*.
	\end{align*}
	This implies that $\Omega_T\in\mathcal{O}_\omega$ is a minimizer of $J_T$.  Similar arguments can be done to show that there exists $\Omega_s\in\mathcal{O}_\omega$ that minimizes $J_s$. 
\end{proof}


\section{Proof of estimate (5)} 
\label{section:5}
In this section, given the recently established existence of shape solutions, we now prove \eqref{ineq:goal}.

\begin{theorem}
	Suppose that ${\blf}\in L^2(\mathcal{D};\mathbb{R}^2)$, ${\bu}_0\in H(\Omega)\cap L^2(\mathcal{D};\mathbb{R}^2)$, and ${\bu}_D \in H^1(\omega;\mathbb{R}^2)$. If $J_T^* := J_T(\Omega_T)$ and $J_s^* := J_s(\Omega_s)$, where $\Omega_T$ and $\Omega_s$ are the solutions of \eqref{problem:timedependent} and \eqref{problem:stationary}, respectively; then there exists $c>0$, independent of $T$, such that \eqref{ineq:goal} holds.
	\label{theorem:goal}
\end{theorem}
\begin{proof}
	Firstly, note that since $J_T^*$ is the minimizer of $J_T$, $J_T^* \le J_T(\Omega_s)$. Hence, $J_T^* - J_s^* \le J_T(\Omega_s) - J_s(\Omega_s)$.  
	
	We consider the solution ${\bu}_s \in  L^\infty(I;H(\Omega_s))\cap L^2(I;V(\Omega_s))$ of \eqref{weak:timedependent} on $\Omega = \Omega_s$.  According to Theorem \ref{theorem:statewp}({\it i}), this solution satisfies \eqref{estimate:Linf} and \eqref{estimate:L2} in $\Omega = \Omega_s$. 
	
	We also consider the element $\delta{\bu}_s(t) : = {\bu}_s(t) - {\bv}_s$ for $t\in [0,T]$, where ${\bv}_s\in V(\Omega_s)$ is the optimal state for the problem \eqref{problem:stationary} that solves \eqref{weak:stationary} in $\Omega = \Omega_s$. This element solves the following variational equation
	\begin{align}
		\begin{aligned}
			&{}_{V^*(\Omega_s)}\langle \partial_t\delta{\bu}_s, {\bphi}\rangle_{V(\Omega_s)} + \nu(\nabla\delta{\bu}_s,\nabla{\bphi})_{\Omega_s} + \gamma(({\bv}_s\cdot\nabla)\delta{\bu}_s,{\bphi})_{\Omega_s}  = \gamma ((\delta{\bu}_s\cdot\nabla){\bphi},{\bv}_s)_{\Omega_s},
		\end{aligned}
		\label{weak:sdelta}
	\end{align}
	for all ${\bphi}\in V(\Omega_s)$, and almost every $t\in [0,T]$. Furthermore, $\delta{\bu}_s$ satisfies the initial condition $\delta{\bu}_s(0) : = {\bu}_0 - {\bv}_s$. By letting ${\bphi} = \delta{\bu}_s$ in \eqref{weak:sdelta}, we get
	\begin{align}
		\begin{aligned}
			\frac{1}{2}\partial_t\|\delta{\bu}_s(t)\|_{H(\Omega_s)}^2 &+ \nu \|\delta{\bu}_s(t)\|_{V(\Omega_s)}^2 = \gamma((\delta{\bu}_s(t)\cdot\nabla)\delta{\bu}_s(t),{\bv}_s)_{\Omega}.
		\end{aligned}
		\label{eqn:diagdelta}
	\end{align}
	{Taking the integral of both sides of \eqref{eqn:diagdelta} over $(0,t)\subset [0,T]$ and employing \eqref{trilinear:ladyzhenskaya}, Poincare inequality, and \eqref{estimate:stationary} yield
		\begin{align*}
			\frac{1}{2}\|\delta{\bu}_s(t)\|_{H(\Omega_s)}^2 &+ \nu\int_0^t\|\delta{\bu}_s(t)\|_{V(\Omega_s)}^2\du t
			= \|{\bu}_0 - {\bv}_s\|_{H(\Omega_s)}^2 + \gamma\int_0^t ((\delta{\bu}_s(t)\cdot\nabla)\delta{\bu}_s(t),{\bv}_s)_{\Omega}\du t\\
			&\le \|{\bu}_0 - {\bv}_s\|_{H(\Omega_s)}^2 + 2^{1/2}\gamma\tilde{c}\int_0^t \|\delta{\bu}_s(t)\|_{V(\Omega_s)}^2 \|{\bv}_s\|_{V(\Omega_s)} \du t\\
			&	\le \|{\bu}_0 - {\bv}_s\|_{H(\Omega_s)}^2 + \frac{2^{1/2}\gamma\tilde{c}^2}{\nu}\|{\blf} \|_{L^2(\mathcal{D};\mathbb{R}^2)}\int_0^t \|\delta{\bu}_s(t)\|_{V(\Omega_s)}^2  \du t
		\end{align*}
		By moving the second term on the right-hand side of the inequality above to the left-hand side, we infer that
		\begin{align}
			\frac{1}{2}\|\delta{\bu}_s(t)\|_{H(\Omega_s)}^2 + \frac{(\nu^2 - 2^{1/2}\gamma\tilde{c}^2\|{\blf}\|_{L^2(\mathcal{D};\mathbb{R}^2) } )}{\nu}\int_0^t\|\delta{\bu}_s(t)\|_{V(\Omega_s)}^2\du t \le \|{\bu}_0 - {\bv}_s\|_{H(\Omega_s)}^2.
			\label{estimate:delta}
		\end{align}
		Notice from \eqref{estimate:delta} that the energy of $\delta{\bu}_s$ is bounded if $ 2^{1/2}\gamma\tilde{c}^2\|{\blf}\|_{L^2(\mathcal{D};\mathbb{R}^2) } < \nu^2$, which is implicitly assumed to assure uniqueness of the solution to \eqref{weak:stationary}.  Furthermore, we also notice that \eqref{estimate:delta} validates item (iv) in Remark \ref{remark:2} regarding the gap between the solution of the time-dependent and stationary Navier--Stokes equations.}
	
	{
		Let us now derive an estimate for $J_T^* - J_s^*$. 
		We use Poincar{\'e} inequality, reverse triangle inequality, and the estimates \eqref{estimate:Linf}, \eqref{estimate:L2}, \eqref{estimate:stationary}, and \eqref{estimate:delta} to get
		\begin{align*}
			J_T^* -  J_s^* \le& \,  J_T(\Omega_s) - J_s(\Omega_s)
			\le  \frac{\nu}{T}\int_0^T  \|\nabla({\bu}_s(t) - {\bu}_D)\|_{L^2(\omega;\mathbb{R}^{2\times 2})}^2 - \|\nabla({\bv}_s - {\bu}_D)\|_{L^2(\omega;\mathbb{R}^{2\times2})}^2\du t\\
			& + \frac{\nu}{T}\int_0^T  \|{\bu}_s(t) - {\bu}_D\|_{L^2(\omega;\mathbb{R}^{2\times 2})}^2 - \|{\bv}_s - {\bu}_D\|_{L^2(\omega;\mathbb{R}^{2\times2})}^2\du t\\
			\le &\,  \frac{\nu}{T}\int_0^T  \big(\|\nabla({\bu}_s(t) - {\bu}_D)\|_{L^2(\omega;\mathbb{R}^{2\times 2})} - \|\nabla({\bv}_s - {\bu}_D)\|_{L^2(\omega;\mathbb{R}^{2\times2})}\big)\\
			& \times  \big(\|\nabla({\bu}_s(t) - {\bu}_D)\|_{L^2(\omega;\mathbb{R}^{2\times 2})} + \|\nabla({\bv}_s - {\bu}_D)\|_{L^2(\omega;\mathbb{R}^{2\times2})}\big)\du t\\
			& + \frac{\nu}{T}\int_0^T  \big(\|{\bu}_s(t) - {\bu}_D\|_{L^2(\omega;\mathbb{R}^{2\times 2})} - \|{\bv}_s - {\bu}_D\|_{L^2(\omega;\mathbb{R}^{2\times2})}\big)\\
			& \times  \big(\|{\bu}_s(t) - {\bu}_D\|_{L^2(\omega;\mathbb{R}^{2\times 2})} + \|{\bv}_s - {\bu}_D\|_{L^2(\omega;\mathbb{R}^{2\times2})}\big)\du t\\
			\le &\,   \frac{c\nu}{T}\|\delta{\bu}_s \|_{L^2(I;V(\Omega_s))}\Big(\int_0^T\|\nabla{\bu}_s(t)\|_{L^2(\omega;\mathbb{R}^2)}^2 \du t + T\|\nabla{\bv}_s \|_{L^2(\omega;\mathbb{R}^2)}^2 +  T\|\nabla{\bu}_D\|_{L^2(\omega;\mathbb{R}^2)}^2 \Big)^{\!1/2}\\
			\le &\,   \frac{c\nu}{T}\|{\bu}_0 - {\bv}_s\|_{H(\Omega_s)}\Big(\frac{(c_2+1)T}{\nu^2}\|{\blf}\|_{L^2(\mathcal{D};\mathbb{R}^2)}^2 + \frac{1}{\nu}\|{\bu}_0\|_{L^2(\mathcal{D};\mathbb{R}^2)}^2   +  T\|\nabla{\bu}_D\|_{L^2(\omega;\mathbb{R}^2)}^2 \Big)^{\!1/2}\\
			\le &\,   \frac{c}{T}\left( C_1 + C_2\sqrt{T} \right),
		\end{align*}
		where $C_1 : = C_1( {\bu}_0, {\blf}, 1/\nu, \mathcal{D} )$, and $C_2 := C_2( {\bu}_D, {\blf}, 1/\nu, \mathcal{D} )$. 
		Hence, we get
		\begin{align*}
			J_T^* - J_s^* \le C\left( \frac{1}{T} + \frac{1}{\sqrt{T}} \right).
	\end{align*} }
	
	Similarly, if we consider the solution $\Omega_T$ of \eqref{problem:timedependent}, then $J_s^* \le J_s(\Omega_T)$, hence $J_s^* - J_T^* \le J_s(\Omega_T) - J_T(\Omega_T)$. Moreover, by considering the optimal state ${\bu}_T \in L^\infty(I;H(\Omega_T))\cap L^2(I;V(\Omega_T))$, and the element ${\bv}_T \in V(\Omega_T)$ that solves \eqref{weak:stationary} in $\Omega = \Omega_T$, we can show, by using the same arguments as in the previous step, that
	\begin{align}
		\begin{aligned}
			J_s^* - J_T^* & \le J_s(\Omega_T) - J_T(\Omega_T) \le c\left( \frac{1}{T} + \frac{1}{\sqrt{T}} \right)
		\end{aligned}
		\label{ineq:omegaT}
	\end{align}
	for some $c : = c( {\bu}_0,{\bu}_D, {\blf}, 1/\nu, \mathcal{D} ) > 0$. This proves inequality \eqref{ineq:goal}. 
\end{proof}


As a consequence, we obtain a sense of convergence of solutions of \eqref{problem:timedependent} to a solution of \eqref{problem:stationary}. We formalize this result in the following corollary.
\begin{corollary}
	Suppose that the assumptions in Theorem \ref{theorem:goal} hold; then there exists $\Omega^*\in\mathcal{O}_\omega$, such that $\Omega_T\chiarrow \Omega^*$ as $T\to\infty$, and that $\Omega^*$ solves \eqref{problem:stationary}.
\end{corollary}
\begin{proof}
	Let $\{T_n\}\subset(0,\infty)$ be a sequence such that $T_n\to\infty$ as $n\to\infty$. For each $n\in\mathbb{N}$, $J_{T_n}$ has a minimizer denoted as $\Omega_{T_n}\in\mathcal{O}_\omega$, hence we obtain a sequence $\{\Omega_{T_n}\}\subset\mathcal{O}_\omega$ and an element $\Omega^*\in\mathcal{O}_\omega$ such that $\Omega_{T_n}\chiarrow\Omega^*$. From \eqref{conv:strongstationary},  ${\bv}_{T_n}\to {\bv}^*$ in $H^1(\omega;\mathbb{R}^2)$, where ${\bv}_{T_n}\in V(\Omega_{T_n})$ solves \eqref{weak:stationary} in $\Omega = \Omega_{T_n}$ and ${\bv}^*\in V(\Omega^*)$ is a solution of \eqref{weak:stationary} in $\Omega=\Omega^*$. 
	
	From \eqref{ineq:goal}, \eqref{ineq:omegaT} and $J_{T_n}^* := J_{T_n}(\Omega_{T_n})$, we now get
	\begin{align*}
		| J_s^* -J_s(\Omega^*)| \le&\, |J_s^* - J_{T_n}^*| + |J_{T_n}(\Omega_{T_n}) - J_s(\Omega_{T_n})| + |J_s(\Omega_{T_n}) - J_s(\Omega^*)|\\
		\le&\, 2c\left( \frac{1}{T_n} + \frac{1}{\sqrt{T_n}}  \right) +  |J_s(\Omega_{T_n}) - J_s(\Omega^*)|.
	\end{align*}
	For the term $|J_s(\Omega_{T_n}) - J_s(\Omega^*)|$, we have the following computations
	\begin{align*}
		|J_s(\Omega_{T_n})& - J_s(\Omega^*)| \le c\nu \left|\|{\bv}_{T_n} - {\bu}_D\|_{H^1(\omega;\mathbb{R}^2)}^2 - \|{\bv}^* - {\bu}_D\|_{H^1(\omega;\mathbb{R}^{2})}^2 \right|\\
		\le &\, c\nu \left|\|{\bv}_{T_n} - {\bu}_D\|_{H^1(\omega;\mathbb{R}^2)} - \|{\bv}^* - {\bu}_D\|_{H^1(\omega;\mathbb{R}^{2})} \right|\times \left|\|{\bv}_{T_n} - {\bu}_D\|_{H^1(\omega;\mathbb{R}^2)} + \|{\bv}^* - {\bu}_D\|_{H^1(\omega;\mathbb{R}^{2})} \right|\\
		\le &\, c\nu \|{\bv}_{T_n} - {\bv}^*\|_{H^1(\omega;\mathbb{R}^2)}\big( \|{\bv}_{T_n}\|_{H^1(\omega;\mathbb{R}^2)} + \|{\bv}^*\|_{H^1(\omega;\mathbb{R}^2)} + \|{\bu}_D\|_{H^1(\omega;\mathbb{R}^2)} \big)
	\end{align*}
	By letting $n\to \infty$ and from \eqref{conv:strongstationary}, the inequality above implies that $|J_s(\Omega_{T_n}) - J_s(\Omega^*)|\to 0$, and therefore \[| J_s^* -J_s(\Omega^*)|\le 0.\] 
\end{proof}


\section{Numerical Realization}\label{section:6}

In this section we illustrate the convergence numerically. For simplicity, we only consider the $L^2$ version of the objective functions, i.e., we consider 
\begin{align}
	\left.
	\begin{aligned}
		\min_{\Omega\in \mathcal{O}_{\omega}} J_{T,\times}(\Omega) := \frac{\nu}{T}&\int_0^T \|{\bu}(t) - {\bu}_D\|_{L^2(\omega;\mathbb{R}^{2})}^2 \du t\quad \text{subject to }\eqref{system:timedependent},
	\end{aligned}
	\right.
	\label{problem:extimedependent}
\end{align}
and 
\begin{align}
	\left.
	\begin{aligned}
		\min_{\Omega\in \mathcal{O}_{\omega}} J_{s,\times}(\Omega) :=\nu\|{\bv} - {\bu}_D\|_{L^2(\omega;\mathbb{R}^{2})}^2\quad \text{subject to }\eqref{system:stationary}.
	\end{aligned}
	\right.
	\label{problem:xstationary}
\end{align}
Note that the analyses done in the previous sections hold true due to Poincar{\'e} inequality.

To solve the problem numerically, we shall rely on a gradient descent method induced by the identity perturbation operator. In particular, for a given domain $\Omega\in\mathcal{O}_{\omega}$, we consider the following family of perturbed domains $\{\Omega_\tau:=T_\tau(\Omega): 0\le\tau\le\tau_0 \}\subset\mathcal{O}_\omega$, where $T_\tau:\mathcal{D}\to \mathbb{R}^2$ is defined as $T_\tau(x) = x + \tau\theta(x)$ for any $x\in \mathcal{D}$ and $\tau\in[0,\tau_0]$, where $\theta\in \Theta :=\{\theta\in C^{1,1}(\mathcal{D};\mathbb{R}^2): \theta = 0 \text{ on }\partial\mathcal{D} \}$ is known as the deformation field and $\tau_0$ is a given threshold parameter so as to make sure that $\{\Omega_\tau:=T_\tau(\Omega): 0\le\tau\le\tau_0 \}\subset\mathcal{O}_\omega$ and that the translated states are well-posed. For more details, we refer to \cite{zolesio2011} or \cite{sokolowski1992}.  We compute the shape derivative of the objective functions in the sense of Hadamard's, that is, the shape derivative of a given objective function $\mathcal{J}:\mathcal{O}_\omega\to\mathbb{R}$ in the direction of $\theta\in\Theta$  is denoted by $d\mathcal{J}(\Omega)\theta$ and is defined as
\[
d\mathcal{J}(\Omega)\theta = \lim_{\tau\searrow0}\frac{\mathcal{J}(\Omega_\tau)-\mathcal{J}(\Omega)}{\tau}.
\]

The shape derivatives of $J_{T,\times}$ and $J_{s,\times}$ have been computed by several authors, hence we skip such step in this exposition. Refer to \cite{gao2008,ito2008,kasumba2013b,pironneau2010} among others for such computations. Nevertheless, we give such derivatives below:{
	\begin{align}
		dJ_{T,\times}(\Omega)\theta &= \frac{\nu}{T} \left[\int_{\partial\Omega}\int_0^T\left( \frac{\partial{\bu}(t)}{\partial{\bn}}\cdot\frac{\partial{\bw(t)}}{\partial{\bn}} \right)\theta\cdot{\bn}\du t\du\sigma  +\int_\Omega\int_0^T \nabla\cdot( \chi_\omega|{\bu}(t) - {\bu}_D|^2\theta)\du t\du x \right],\\
		dJ_{s,\times}(\Omega)\theta &={\nu} \int_{\partial\Omega}\left( \frac{\partial{\bv}}{\partial{\bn}}\cdot\frac{\partial{\bz}}{\partial{\bn}} \right)\theta\cdot{\bn}\du\sigma  +\int_\Omega \nabla\cdot( \chi_\omega|{\bv} - {\bu}_D|^2\theta)\du x ,
\end{align}}
where ${\bw}\in L^\infty(I;H(\Omega))\cap L^2(I;V(\Omega))$ is the adjoint variable for the time dependent problem  that solves the following variational problem
\begin{align}
	&\begin{aligned}
		{}_{V^*(\Omega)}\langle -\partial_t{\bw}(t), {\bphi}\rangle_{V(\Omega)} + &\nu(\nabla{\bw}(t),\nabla{\bphi})_\Omega + \gamma(({\bphi}\cdot\nabla){\bu}(t),{\bw}(t))_{\Omega}\\ &- \gamma(({\bu}(t)\cdot\nabla){\bw}(t),{\bphi})_{\Omega} = 2({\bu}(t)-{\bu}_D,{\bphi})_{\omega} \quad\forall{\bphi}\in V(\Omega),
	\end{aligned}
	\label{weak:timeadjoint}
\end{align}
and satisfies the transversality condition $w(T) = 0$, while ${\bz}\in V(\Omega)$ solves the following weak equation
\begin{align}
	& \nu(\nabla{\bz},\nabla{\bphi})_\Omega + \gamma(({\bphi}\cdot\nabla){\bv},{\bz})_{\Omega}- \gamma(({\bv}\cdot\nabla){\bz},{\bphi})_{\Omega} = 2({\bv}-{\bu}_D,{\bphi})_{\omega} \quad\forall{\bphi}\in V(\Omega).
	\label{weak:adjoint}
\end{align}

Note that we can express both derivatives in the form of the Zolesio-Hadamard structure, i.e.,
\[
d\mathcal{J}(\Omega)\theta = \int_{\partial\Omega} \nabla J{\bn}\cdot\theta \du\sigma,
\]
where $\nabla J$ is called the shape gradient. In particular, by virtue of the divergence theorem, we get the following shape gradients
\begin{align*}
	\nabla J_{T,\times} &= \frac{\nu}{T}\int_0^T \frac{\partial{\bu}(t)}{\partial{\bn}}\cdot\frac{\partial{\bw(t)}}{\partial{\bn}} + \chi_\omega|{\bu}(t) - {\bu}_D|^2 \du t,\\
	\nabla J_{s,\times} &= {\nu}\left[ \frac{\partial{\bv}}{\partial{\bn}}\cdot\frac{\partial{\bz}}{\partial{\bn}} + \chi_\omega|{\bv} - {\bu}_D|^2 \right],
\end{align*}
for the time-dependent and stationary objective functions, respectively. These shape gradients will be the basis of our descent directions, that is, by choosing $\theta = -\nabla J{\bn}$ in $\partial\Omega$ we are assured that 
\[
d\mathcal{J}(\Omega)\theta = -\|\theta\|^2_{L^2(\partial\Omega;\mathbb{R}^2)}.
\]
Numerically though, such choice of descent direction may cause oscillations on the perturbed domains $\Omega_\tau$. Because of that, we shall resort to a traction method that intends to extend the choice $\theta = -\nabla J{\bn}$ to the whole domain, say for example by a Robin boundary problem which we shall briefly discuss later.

The variational equations are solved using Galerkin finite element methods. For the nonlinearity of the stationary Navier--Stokes equations, we employ Newton's method \cite{girault1986}, while for the dynamic Navier-Stokes equations and the time-dependent adjoint equations we utilize a Lagrange-Galerkin method based on characteristics (see for example \cite{notsu2016}).  Since the stationary adjoint equation is a linear system, we utilize the usual Galerkin method.

For the resolution of the deformation fields, we shall utilize an $H^1$-gradient based method \cite{azegami2006}. For the deformation field of the stationary problem, we solve the following variational problem:
\begin{align}
	\varepsilon (\nabla\theta,\nabla\bphi)_\Omega + (\theta,\bphi)_{\partial\Omega} = -(\nabla J_{s,\times}{\bn},{\bphi})_{\partial\Omega}\quad \forall{\bphi}\in H^1_0(\mathcal{D};\mathbb{R}^2).
	\label{deform:stationary}
\end{align}
Meanwhile, due to the time-dependent nature of the shape gradient of \eqref{problem:timedependent} we propose a time-averaged deformation field, in particular, by letting $K(\bu,\bw)(t)$ be such that $\nabla J_{T,\times} =\frac{1}{T} \int_0^T K(\bu,\bw)(t)\du t$, then we aim to solve for $\vartheta(t)\in H^1_0(\mathcal{D};\mathbb{R}^2) $ that satisfies, for any $t\in[0,T]$, the equation
\begin{align}
	\varepsilon (\nabla\vartheta(t),\nabla\bphi)_\Omega + (\vartheta(t),\bphi)_{\partial\Omega} = -(K(\bu,\bw)(t){\bn},{\bphi})_{\partial\Omega}\quad \forall{\bphi}\in H^1_0(\mathcal{D};\mathbb{R}^2).
	\label{deform:instationary}
\end{align}
From these time-dependent {\it vector fields}, we then determine the deformation field by $\theta = \frac{1}{T}\int_0^T\vartheta\du t$.

Note that in both equations \eqref{deform:stationary} and \eqref{deform:instationary}, $\varepsilon>0$ can be chosen small enough so that $\theta \approx -\nabla J_{s,\times}{\bn}$ and $\vartheta(t)\approx -K({\bu},{\bw})(t){\bn}$ on $\partial\Omega$.

\subsection{Finite-Dimensional Approximation Schemes and Optimization Algorithms} 
The variational problems, as previously mentioned, will be solved using finite element methods.  Let $\mathcal{T}_h = \{K\}$ be a regular triangulation of a domain $\Omega\in\mathcal{O}_\omega$ so that $\cup_{\mathcal{T}_h} K =:\Omega_h\subset\Omega $ and that there exists $\mathcal{T}_{h,\omega}\subset\mathcal{T}_h$ such that $\cup_{\mathcal{T}_{h,\omega}} K=:\omega_h$ is a discretization of $\omega$, and $\mathbb{P}^k(K;\mathbb{R}^d)$ be the space of $k^{th}$ degree polynomials from $K$ onto $\mathbb{R}^d$, we consider the following space Taylor-Hood finite element spaces
\begin{align*}
	X_h &:= \{{\bv}_h\in C(\overline{\Omega};\mathbb{R}^2): {\bv}_h|_{K}\in \mathbb{P}^2(K;\mathbb{R}^2),\, \forall K\in\mathcal{T}_h \},\\
	M_h &:= \{{q}_h\in C(\overline{\Omega};\mathbb{R}): {\bv}_h|_{K}\in \mathbb{P}^1(K;\mathbb{R}),\, \forall K\in\mathcal{T}_h \},
\end{align*}
$V_h := X_h \cap H_0^1(\Omega;\mathbb{R}^2)$, $W_h:= X_h\cap H^1(\Omega;\mathbb{R}^2)$, $Q_h:= M_h\cap L^2(\Omega;\mathbb{R})$, and $Y_h := V_h\times M_h$.

For approximating the solution to the state equation \eqref{weak:stationary}, let $\mathcal{F}_h:Y_h\to Y_h^*$ be an operator defined as 
\begin{align*}
	{}_{Y_h^*}\langle \mathcal{F}_h({\bv},q),({\bphi},\psi)\rangle_{Y_h} := \nu(\nabla{\bv},\nabla{\bphi})_{\Omega_h}& + \gamma(({\bv}\cdot\nabla){\bv},{\bphi})_{\Omega_h} - (\nabla\cdot{\bv},\psi)_{\Omega_h} - (\nabla\cdot{\bphi},q)_{\Omega_h}  - ({\blf},{\bphi})_{\Omega_h}.
\end{align*}
Since $b({\bv},\psi) := -(\nabla\cdot{\bv},\psi)_{\Omega_h} $ satisfies the inf-sup condition in our particular choice of finite element spaces, we are assured of the existence of the solution $({\bv}_h,q_h)\in Y_h$ to $\mathcal{F}_h({\bv}_h,q_h) = 0 $ in $Y_h^*$. Furthermore,  the first coordinate of the solution $({\bv}_h,q_h)\in Y_h$ satisfies \eqref{weak:stationary} in its discretized form. Of course, the integrals on the discrete domain $\Omega_h$ are understood as discrete integrals, say for example by Gaussian quadrature.

Due to the nonlinearity of $\mathcal{F}_h$, its root will be approximated using Newton's scheme, so that for an initial point $({\bv}^0_h,q^0_h)\in Y_h$, a sequence $\{({\bv}^k_h,q^k_h)\}\subset Y_h$ is generated by solving the equation
\begin{align}
	({\bv}^{k+1}_h,q^{k+1}_h) = ({\bv}^k_h,q^k_h) - [\mathcal{F}_h'({\bv}^k_h,q^k_h)]^{-1}\cdot \mathcal{F}_h({\bv}^k_h,q^k_h),
	\label{newton:dualform}
\end{align}
or equivalently, by writing $(\delta{\bv}^k_h,\delta q^k_h) = ({\bv}^{k+1}_h,q^{k+1}_h)-({\bv}^k_h,q^k_h)$,
\begin{align}
	{}_{Y_h^*}\langle \mathcal{F}_h'({\bv}^k_h,q^k_h)(\delta{\bv}^k_h,\delta q^k_h),({\bphi},\psi)\rangle_{Y_h} = - {}_{Y_h^*}\langle \mathcal{F}_h({\bv}^k_h,q^k_h),({\bphi},\psi)\rangle_{Y_h}
	\label{newton:primalform}
\end{align}
The existence of solutions to \eqref{newton:dualform} and \eqref{newton:primalform} are due to the isomorphism of $\mathcal{F}_h\in\mathcal{L}(Y_h,Y_h^*)$ which is a consequence of the uniqueness assumption $2^{1/2}\gamma \tilde{c}^2\|{\blf}\|_{L^2(\mathcal{D};\mathbb{R}^2)}<\nu^2$. We end the generation of the elements of the sequence $\{({\bv}^k_h,q^k_h)\}\subset Y_h$ when we reach a certain tolerance which is measured by $\|\delta{\bv}^k \|_{V(\Omega)}/\|{\bv}^k \|_{V(\Omega)}$.

The approximation ${\bz}_h\in V_h$ of the adjoint variable ${\bz}\in V(\Omega)$, on the other hand, is done by solving the variational equation
\begin{align}
	{}_{Y_h^*}\langle \mathcal{F}_h'({\bv}^{\star}_h,q^{\star}_h)({\bphi},\psi),({\bz}_h,\pi_h)\rangle_{Y_h} = 2({\bv}_h^{\star} - {\bu}_{D},{\bphi})_{\omega_h} \quad \forall({\bphi},\psi)\in Y_h,
	\label{newton:adjoint}
\end{align}
where $({\bv}^{\star}_h,q^{\star}_h)\in Y_h$ is the approximation of the solution of \eqref{weak:stationary} yielding from Newton's scheme, and $\pi_h\in Q_h$ is the adjoint pressure.

For the time-dependent problems, we utilize Lagrange-Galerkin methods. An upwind Lagrange-Galerkin method is intended to solve the state equation \eqref{weak:timedependent} while a downwind method is for the adjoint equation \eqref{weak:timeadjoint}. Let us define the material derivative $D/Dt$ by
\begin{align*}
	\frac{D[\bullet]}{Dt} := \frac{\partial[\bullet]}{\partial t} + \gamma({\bu}\cdot\nabla)[\bullet].
\end{align*}
We shall consider characteristic lines that solve the differential equation
\begin{align}
	\frac{d\boldsymbol{x}}{dt} = \gamma{\bu}(\boldsymbol{x}(t),t),
	\label{ode:char}
\end{align}
so that for sufficiently smooth ${\bu}$, and velocity field ${\bw}:\Omega\times[0,T]\to\mathbb{R}^2$
\begin{align*}
	\frac{D{\bw}}{Dt} = \frac{d}{dt}{\bw}(\boldsymbol{x}(t),t).
\end{align*}

Let $\Delta t = T/N$ be a time increment over $N$ subintervals of the interval $[0,T]$,  $t^n:=n\Delta$, and $h:\Omega\times[0,T]\to \mathbb{R}^d$, we denote the evaluation $h(\cdot,t^n)$ by $h^n$. Let $x\in \mathbb{R}^2$. The solution to \eqref{ode:char} with initial value $\boldsymbol{x}(t^n) = x$ will be denoted as $\boldsymbol{x}(\cdot;x,t^n)$. From a velocity ${\bu}$, we will utilize the upwind point of $x\in\mathbb{R}^2$ with respect to ${\bu}$ which is defined as and denoted by $\boldsymbol{x}_u({\bu},\Delta t)(x) : = x-{\bu}\Delta t$, and the downwind point of $x$ with respect to ${\bu}$ defined as and denoted by $\boldsymbol{x}_d({\bu},\Delta t)(x) := x + {\bu}\Delta t$. From these directions, we have the following approximations
\begin{align*}
	\boldsymbol{x}_u({\bu}^{n-1},\Delta t) & \approx \boldsymbol{x}(t^{n-1};x,t^n),\\
	\boldsymbol{x}_d({\bu}^{n+1},\Delta t) & \approx \boldsymbol{x}(t^{n+1};x,t^n).
\end{align*}
We can then consider a first order forward in-time approximation of the material derivative at $(x,t^n)$ by
\begin{align*}
	\frac{D \bw}{D t}(x,t^n) 	& = \frac{d}{dt}{\bw}(\boldsymbol{x}(t;x,t^n),t)\big|_{t = t_n}\\
	& = \frac{{\bw}(\boldsymbol{x}(t^n;x,t^n),t^n) - {\bw}(\boldsymbol{x}(t^{n-1};x,t^n),t^{n-1})}{\Delta t}\\
	& \approx \frac{{\bw}^n - {\bw}^{n-1}\circ\boldsymbol{x}_u({\bu}^{n-1},\Delta t)}{{\Delta t}}(x).
\end{align*}
Meanwhile, for the first order backward in-time approximation, we have
\begin{align*}
	\frac{D \bw}{D t}(x,t^n) 	& = \frac{d}{dt}{\bw}(\boldsymbol{x}(t;x,t^n),t)\big|_{t = t_n}\\
	& = \frac{{\bw}(\boldsymbol{x}(t^{n+1};x,t^n),t^{n+1}) - {\bw}(\boldsymbol{x}(t^{n};x,t^n),t^{n})}{\Delta t}\\
	& \approx \frac{{\bw}^{n+1}\circ\boldsymbol{x}_d({\bu}^{n+1},\Delta t) - {\bw}^n}{{\Delta t}}(x).
\end{align*}

From these, we can approximate the solutions to \eqref{weak:timedependent} and \eqref{weak:timeadjoint} as follows. For the Navier--Stokes equations \eqref{weak:timedependent}, by letting $(\bu_h^0,p_h^0)\in Y_h$ the projection of $(\bu_0,0)\in (H(\Omega)\cap L^2(\mathcal{D};\mathbb{R}^2))\times L^2_0(\Omega;\mathbb{R})$, we generate the sequence $\{({\bu}_h^n,p_h^n)\}_{n=1}^N\subset Y_h$ that satisfies, for each $n = 1,2,\ldots,N$, the equation
\begin{align}
	\begin{aligned}
		\Big(\frac{{\bu}_h^n - {\bu}_h^{n-1}\circ\boldsymbol{x}_u({\bu}_h^{n-1},\Delta t)}{{\Delta t}},&\bphi  \Big)_{\Omega_h} +  \nu(\nabla{\bu}_{h}^n,\nabla{\bphi})_{\Omega_h}\\ & - (\nabla\cdot{\bphi},p_h^n )_{\Omega_h} - (\nabla\cdot{\bu}_h^n,\psi)_{\Omega_h} = ({\blf},{\bphi})_{\Omega_h}
	\end{aligned}\quad\forall({\bphi},\psi)\in Y_h.
	\label{LG:navstokes}
\end{align} 
While the adjoint equation \eqref{weak:timeadjoint} is approximated in a backward manner, so that for $({\bw}_h^N,\pi_h^N) = (0,0)\in Y_h$, the sequence $\{({\bw}_h^m,\pi_h^n)\}_{n=0}^{N-1}\subset Y_h$ is generated by the difference equation
\begin{align}
	\begin{aligned}
		\Big(\frac{{\bw}_h^n - {\bw}_h^{n+1}\circ\boldsymbol{x}_d({\bu}_h^{n+1},\Delta t)}{{\Delta t}}&,\bphi  \Big)_{\Omega_h} + \nu (\nabla{\bw}_h^n,\nabla{\bphi})_{\Omega_h} + \gamma([\nabla{\bu}_h^n]^\top{\bw}_h^n,{\bphi} )_{\Omega_h}\\ 
		-& (\nabla\cdot{\bphi},\pi_h^n )_{\Omega_h} - (\nabla\cdot{\bw}_h^n,\psi)_{\Omega_h} = 2({\bu}_h^n - {\bu}_{D},{\bphi})_{\omega_h}
	\end{aligned}\ \ \forall({\bphi},\psi)\in Y_h.
	\label{LG:adjoint}
\end{align}

Due to the linear nature of the Robin problems \eqref{deform:stationary} and \eqref{deform:instationary}, we can solve them quite easily. In particular, for the approximation of the deformation field of the stationary problem, we have the following discretized equation:
\begin{align}
	\varepsilon(\nabla\theta_h,\nabla\varphi)_{\Omega_h} + (\theta_h,\varphi)_{\partial\Omega_h} = - (\nabla J_{s,\times,h},\varphi)_{\partial\Omega_h} \quad\forall \varphi\in W_h,
	\label{discretedeform:stationary}
\end{align}
where $\nabla J_{s,\times,h}$ is the evaluation of the shape gradient $\nabla J_{s,\times}$ at the discrete solutions ${\bv}_h^\star$, ${\bz}_h$. Similarly, we solve the Robin problems, for each $t^n = n\Delta$ (same as the time discretization previously discussed), as 
\begin{align}
	\varepsilon(\nabla\vartheta_h(t^n),\nabla\varphi)_{\Omega_h} + (\vartheta_h(t^n),\varphi)_{\partial\Omega_h} = -(K({\bu}_h^n,{\bw}_h^n)(t^n),\varphi)_{\partial\Omega_h}\quad \forall\varphi\in W_h.
	\label{discretedeform:instationary}
\end{align}
We then solve the time-averaged deformation field for the time-dependent problem using a trapezoidal rule given by
\begin{align}
	\theta_h = \frac{1}{N}\left( \frac{1}{2}\vartheta_h(t^0) + \sum_{k=1}^{N-1}\vartheta_h(t^k) + \frac{1}{2}\vartheta_h(t^N)\right).
	\label{deformtrap}
\end{align}

For the choice of the gradient descent step size, we utilize an Armijo-Goldstein-type line search method. In particular, for a general objective function $\mathcal{J}$ (may it be the stationary or the time-dependent objective function) with the deformation field $\theta$,  for some $\alpha\in(0,1]$ we initially choose the step size
\begin{align}
	\tau = \alpha\frac{\mathcal{J}(\Omega)}{\|\theta\|_{L^2(\partial\Omega;\mathbb{R}^2)}}.
	\label{stepsize}
\end{align}
Since this choice of step size is not sufficient to ensure the descent of the objective function, we employ a backtracking scheme, i.e.,  we choose the smallest $i\in\mathbb{N}$ so that $\mathcal{J}(T_{i}(\Omega)) < \mathcal{J}(\Omega)$ and that $\omega\subset T_i(\Omega)$, where $T_{i} (x) = x + [(0.5)^i\tau]\theta(x)$.

With the ingredients presented above, we lay down the iterative scheme upon which we solve the minimization problems. For the stationary problem, we have the following algorithm\footnote{The steps except that of {\bf Step 0} are inside a \texttt{for loop}.}:
\begin{pethau}
	\item[Step 0.] Choose an initial guess $\Omega_h^0$, and determine the solution ${\bv}_{h}^{\star}(\Omega_h^0)$ of the state equation via Newton's scheme \eqref{newton:primalform} in $\Omega_h^0$.
	\item[Step 1.] Evaluate  $J_{s,\times}(\Omega_h^k)$, and solve for the adjoint variable ${\bz}_h(\Omega_h^k)$ in $\Omega_h^k$ from \eqref{newton:adjoint} and the deformation field $\theta_h(\Omega_h^k)$ from \eqref{discretedeform:stationary} in $\Omega_h^k$;
	\item[Step 2.] Update the domain by $\Omega_h^{k+1} := T_i(\Omega_k)$, solve for the state solution ${\bv}_h^{\star}(\Omega_{k+1})$ from Newton's scheme \eqref{newton:primalform} in $\Omega_{k+1}$, and evaluate $J_{s,\times}(\Omega_h^{k+1})$. 
	\item[Step 3.] If $J_{s,\times}(\Omega_h^{k+1})<J_{s,\times}(\Omega_h^k)$ accept $\Omega_h^{k+1}$ as the new domain, else increase the value of $i\in \mathbb{N}$ and repeat {\bf Step 2}.
\end{pethau}

For the time dependent problem, let us first discuss the method by which we evaluate the objective function $J_{T,\times}$. In fact, we shall use a trapezoidal rule, i.e.,  using the same time discretization as above and a discretized domain $\Omega_h$ we have
\begin{align}
	J_{T,\times}(\Omega_h) \approx \frac{1}{N}\left( \frac{1}{2} \int_{\omega_h} |{\bu}_h^0 - {\bu}_D|^2 \du x + \sum_{k=1}^{N-1} \int_{\omega_h} |{\bu}_h^k - {\bu}_D|^2 \du x + \frac{1}{2}\int_{\omega_h} |{\bu}_h^N - {\bu}_D|^2 \du x\right),
\end{align}
where $\{{\bu}_h^k\}_{k=0}^N$ is the Lagrange--Galerkin approximation of the Navier--Stokes solution from \eqref{LG:navstokes}. From these, we present the following algorithm:
\begin{pethau}
	\item[Step 0.] Choose an initial guess $\Omega_h^0$, and determine the solution $\{{\bu}_{h}^j(\Omega_h^0)\}_{j=0}^N$ of the state equation from \eqref{LG:navstokes} in $\Omega_h^0$.
	\item[Step 1.] Evaluate  $J_{T,\times}(\Omega_h^k)$, and solve for the adjoint variable $\{{\bw}_h^j(\Omega_h^k)\}_{j=0}^N$ in $\Omega_h^k$ from \eqref{LG:adjoint} and the deformation field $\theta_h(\Omega_h^k)$ from \eqref{discretedeform:instationary} and \eqref{deformtrap};
	\item[Step 2.] Update the domain by $\Omega_h^{k+1} := T_i(\Omega_k)$, solve for the state solution $\{{\bu}_{h}^j(\Omega_h^0)\}_{j=0}^N$ from \eqref{LG:navstokes} in $\Omega_{k+1}$, and evaluate $J_{s,\times}(\Omega_h^{k+1})$. 
	\item[Step 3.] If $J_{s,\times}(\Omega_h^{k+1})<J_{s,\times}(\Omega_h^k)$ accept $\Omega_h^{k+1}$ as the new domain, else increase the value of $i\in \mathbb{N}$ and repeat {\bf Step 2}.
\end{pethau}

\subsection{Numerical Implementation}
The finite element problems are solved using \texttt{FreeFem++} \cite{hecht2012} on an Intel Core i7 CPU $@$ 3.80 GHz with 64GB RAM and the codes are stored in the repository \url{https://github.com/jhsimon/NSShapeOptiLongTime}.  For simplicity, we choose the source function to be ${\blf} = \frac{1}{10}(y^3,-x^3)$, the desired function ${\bu}_D$ is determined by solving the stationary Stokes version of the state equations (i.e., with $\gamma=0$) with $\nu = 1/5$ in a domain enclosed in a circle that satisfies $x^2 + y^2 = 4$, and the domain $\omega\subset\mathbb{R}^2$ is the set $\{(x,y)\in\mathbb{R}^2: x^2+y^2\le 1 \}$.  The shape optimization problems \eqref{problem:extimedependent} and \eqref{problem:xstationary} - including the state equations \eqref{weak:timedependent} and \eqref{weak:stationary} that respectively constrain them- are then solved with the assumption that $\nu = \gamma = 1$.  

Due to the uniqueness assumption for the stationary Navier--Stokes equations, we know that the solution yielding from Newton's scheme is a local solution in a branch of nonsingular solutions around the trivial solution ${\bv} \equiv 0$ \cite{girault1986}.  From this point of view, we can choose the initial velocity of the time-dependent Navier--Stokes to be  ${\bu}_0=0$.

For the resolution of the deformation fields, the value $\varepsilon=0.05$ is chosen for both \eqref{deform:stationary} and \eqref{deform:instationary}, and the value $\alpha = 1$ is used for the step size coefficient in \eqref{stepsize}.  For simplicty, the initial domain $\Omega^0$ is defined as the region bounded by the ellipse $9x^2 + 4y^2 = 36$ and is discretized with constant diameter $h = 1/10$. We also mention that due to the tendency of the deformed domains to be degenerate, we employ a mesh refinement at the end of each iterative loop so that the new domain is regular with diameter $h = 1/10$.  Lastly, we terminate the optimization loop when $[\mathcal{J}(\Omega_h^{k+1})-\mathcal{J}(\Omega_h^{k}) ] /\mathcal{J}(\Omega_h^{k+1}) < \texttt{tol} = 1\times 10^{-6}$.

For the stationary problem, Figure \ref{fig1}(A) shows the evolution of the boundary $\partial\Omega$ which is initially chosen as an ellipse and turned into a circle with radius of approximately 3.25 units after 20 iterations which we denote as $\partial\Omega_h^{20} = \partial\Omega_{s,h}$. Figure \ref{fig1}(B) on the other hand shows the decreasing trend of the objective functional on each iterate.

\begin{figure}[h]
 \centering
  \includegraphics[width=1\textwidth]{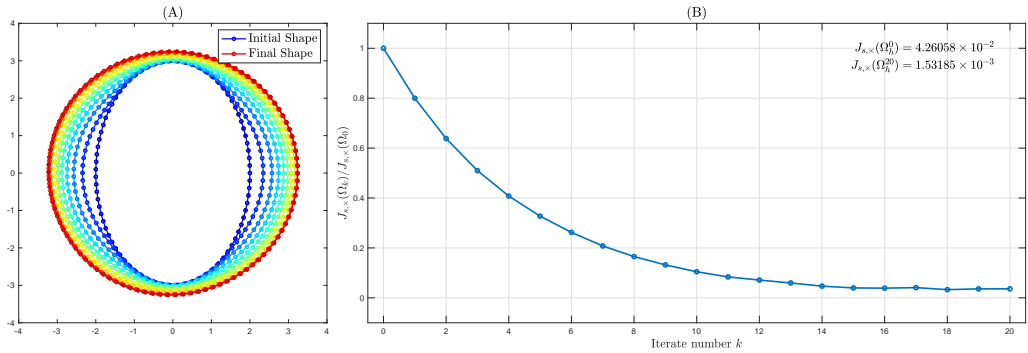}
 \caption{ Evolution of the boundary $\partial\Omega$ from the iterative scheme (A), and the normalized trend of the objective function values at each iteration (B). }
 \label{fig1}
 \end{figure}

To show the convergence of the solutions of the time dependent problems, we implement the numerical simulations with varying terminal time given as $T = 1,2,4,8,16,32,64,$ and $128$, upon which the time discretization is done with a fixed time increment $\Delta t = 0.2$.  For each final time $T$, we denote the final solution as $\Omega_{T,h}$, with boundary $\partial\Omega_{T,h}$. We compare the numerical final solutions in Figure \ref{fig2}(A), where it can be seen that the boundary of the solutions $\partial\Omega_{T,h}$ becomes closer to the boundary $\partial\Omega_{s,h}$ as the terminal time $T$ gets bigger. Figure \ref{fig2}(B) shows the log-log plot of the gap $|J_{T,\times}-J_{s,\times}|$ versus the terminal time $T$. In the same figure, we plotted the log-log plots of $\mathcal{O}(T^{-1})$ and $\mathcal{O}(T^{-1/2})$ to have a gauge on the experimental order of convergence. As expected, we can see that for lower values of $T$, the order of convergence nearly follows that of $\mathcal{O}(T^{-1/2})$, while for the higher values of $T$ we observe a convergence that is similar with that of $\mathcal{O}(T^{-1})$.

\begin{figure}[h]
 \centering
  \includegraphics[width=1\textwidth]{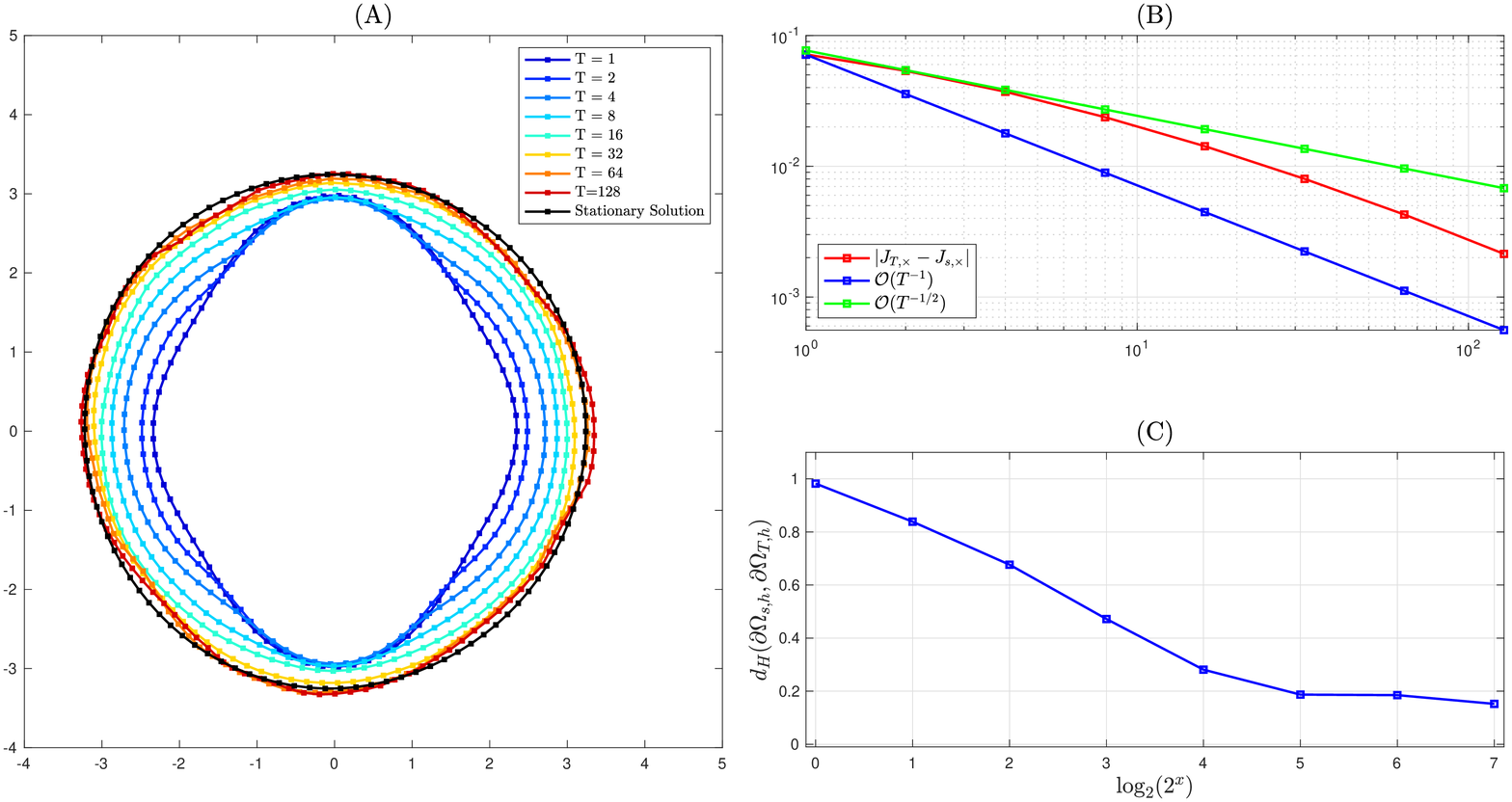}
 \caption{ Illustration of how the boundary of the shape solution of the time-dependent problem \eqref{problem:extimedependent} converges to the boundary of the solution of the equilibrium problem \eqref{problem:xstationary} as $T$ gets larger (A); log-log plots of $|J_{T,\times}-J_{s,\times}|$,  $\mathcal{O}(T^{-1})$, and $\mathcal{O}(T^{-1/2})$(B); trend of the Hausdorff distance between the solutions of \eqref{problem:extimedependent} and \eqref{problem:xstationary} (C). }
 \label{fig2}
 \end{figure}

Lastly, we quantified the convergence of the boundaries $\partial\Omega_{T,h}$ to the boundary $\partial\Omega_{s,h}$ by the virtue of the Hausdorff distance, which is - for any set $A,B$ - denoted and computed as 
\begin{align}
	d_{H}(A,B) : = \max\left\{\sup_{x\in A}d(x,B),\sup_{y\in B}d(A,y) \right\}
\end{align}
where the distance between a set $B$ and a point $x$ is defined as $d(x,B) = \displaystyle\inf_{y\in B}d(x,y)$, and where $d(x,y)$ denotes the distance between $x$ and $y$, which in this case is the usual Euclidean distance. From Figure \ref{fig2}(C), we observe that the Hausdorff distance indeed gets smaller as the value of $x$, which is such that $T = 2^x$, increases.

Before we end this section, let us point out that due to the choice of the viscosity constant $\nu$ in the resolution of the state equations, the advective effects on the fluid is almost negligible. In short, the flow almost mimics that of the Stokes equations. For this reason, we observe an {\it almost full convergence} which is theoretically true for the Stokes versions, and in fact the elliptic/parabolic versions, of our shape design problems.

\section{Conclusion}

In this work, we were able to establish an estimate for the gap between the minimum value of the objective functionals of the dynamic and stationary problems.  In particular, we were able to show that the gap decreases as the time horizon gets large. Furthermore,  we established that as the time horizon goes to infinity the shape solution converges to a domain that is around the neighborhood of a solution of the stationary problem. Lastly, we numerically illustrated the convergence by virtue of the traction method. Here, we have shown that visually, and quantitatively -- by measuring the gap between the optimal value of the stationary and the time-dependent shape design problems with varying terminal time $T$ and with the help of the Hausdorff distance -- the solutions to the dynamic problem indeed converge to the solution of the stationary problem as the time horizon gets bigger.

As mentioned before, we analyzed the systems where the source function is independent of the time variable. Nevertheless, one can also study a more complex problem where the said function is time-dependent but one should also assume some convergence assumption as $T\to\infty$. One can also impose such assumptions on the desired velocity ${\bu}_D$. Lastly, one can also study where the shape also depends on time, upon which the turnpike property now takes place. Such property has never been theoretically proven for shape design problems but has been numerically illustrated by \cite{lance2019}.

\section*{Acknowledgments}
This work was supported by the Japanese Government (MEXT) Scholarship during the course of the author's doctoral studies. The author would also like to acknowledge Professor Hirofumi Notsu for his insights during the course of this work.


\end{document}

%% file: tmp_main_header.tex
\title{Long-Time Behaviour of Shape Design Solutions for the Navier--Stokes Equations}

\author{John Sebastian H. Simon\thanks{Institute of Mathematics,
              Czech Academy of Science,
              {{\v Z}itn{\'a} 25 115 67 Praha 1}, {Czech Republic}} \\(\email{simon@math.cas.cz},\email{jhsimon1729@gmail.com}).}

\headers{Long-Time Behaviour of Shape Design Solutions for the Navier--Stokes Equations}{John Sebastian H. Simon}

%% file: tmp_main_abstract.tex
\begin{abstract}
  We investigate the behavior of dynamic shape design problems for fluid flow at large time horizon.  In particular, we shall compare the shape solutions of a dynamic shape optimization problem with that of a stationary problem and show that the solution of the former approaches a neighborhood of that of the latter. The convergence of domains is based on the $L^\infty$-topology of their corresponding characteristic functions which is closed under the set of domains satisfying the {\it cone property}.  As a consequence, we show that the asymptotic convergence of shape solutions for parabolic/elliptic problems is a particular case of our analysis. Lastly, a numerical example is provided to show the occurrence of the convergence of shape design solutions of time-dependent problems with different values of the terminal time $T$ to a shape design solution of the stationary problem.
\end{abstract}

\begin{keywords}
  Navier-Stokes equations,  long-time behaviour, shape design
\end{keywords}

\begin{AMS}
  49Q10, 49J20, 49K20, 35Q93
\end{AMS}